\documentclass{amsart}

\usepackage{amsmath,latexsym,amscd,amsbsy,amssymb,amsfonts,amsthm,fleqn,leqno,
euscript, graphicx, color}

\newcommand{\Hom}{\mbox{{\rm Hom}}}

\newtheorem{thm}{Theorem}[section]
\newtheorem{pro}[thm]{Proposition}
\newtheorem{lem}[thm]{Lemma}
\newtheorem{cor}[thm]{Corollary}
\newtheorem{ex}[thm]{Example}
\newtheorem{que}[thm]{Question}

\newtheorem{df}[thm]{Definition}
 
\newtheorem{rem}[thm]{Remark}

\newcommand{\cl}{\mbox{{\rm cl}}}
\newcommand{\st}{\mbox{{\rm St}}}
\newcommand{\pr}{\mbox{{\rm pr}}}
\newcommand{\id}{\mbox{{\rm id}}}

\newcommand{\St}{\mbox{{\rm st}}}
\newcommand{\aut}{\mbox{{\rm Aut}}}

\begin{document}

\title{Enveloping semigroups as compactifications of topological groups}

\author{K.\,L.~Kozlov}\thanks{The study of the first author was carried out with the help of the Center of Integration in Science, Ministry of Aliyah and Integration, Israel,  8123461, and is supported by ISF grant 3187/24.}
\address{Department of Mathematics, Ben-Gurion University of the Negev, Beer Sheva, Israel}
\email{kozlovk@bgu.ac.il}
\author{B.\,V.~Sorin}
\address{Lomonosov Moscow State University, Moscow, Russia}
\email{bvs@improfi.ru}


\date{}

 \maketitle
 
\begin{abstract}  An {\it Ellis compactification} $b G$ of $G$ is a proper $G$-compactification of $G$ such that $(bG, \bullet)$ is a right topological monoid and $\bullet|_{G\times b G}=\tilde\alpha$, where $\tilde\alpha$ is the extension of left action of $G$ on itself to $b G$. A method that uses Ellis's "functional approach" is proposed for constructing Ellis compactifications. The proposed method allows us to describe Ellis compactifications of a group $G$ and the corresponding totally bounded uniformities on $G$, to draw conclusions about Roelcke precompactness of $G$, the presence of a semigroup structure on the Roelcke compactification of $G$, and to compare compactifications.

As examples, the proposed method is applied to permutation groups of infinite sets and automorphism groups of ultrahomogeneous chains and LOTS.

\medskip
 

\medskip

\end{abstract}
 
\section{Introduction} 

Compactifications of topological (semi)groups are a vital tool of investigations in Topological dynamics~\cite{Vries}, \cite{GlasnerMegr} and Topological Algebra~\cite{HS}. They play a crucial role in understanding the structure and properties of topological groups and there is an interesting connection between Roelcke compactifications, on the one side, and modern representation theory of ''big'' groups, on the other. They are fundamental in topological dynamics, particularly in studying the behavior of dynamical systems on compact spaces. 

R.~Ellis introduced the "functional approach" in the study of compactifications of groups. A topological group $G$ is a subgroup of the group of homeomorphisms $\Hom(X)$ of a compact space $X$. Each $g\in G$ is examined as a self-map of $X$ and an element of $X^X$ (composition of maps in $X^X$ corresponds to multiplication in $G$). The closure of $G$ in $X^X$ is named the {\it enveloping {\rm(}Ellis{\rm)} semigroup}~\cite{Ellis}. It is a compact right topological monoid and is a non-proper $G$-compactification of $G$. The Ellis enveloping semigroup is proved to be a powerful tool in the abstract theory of topological dynamical systems.  However, explicit computations of enveloping semigroups are quite rare.

In the paper, by a proper compactification of a topological space $X$, we understand a compact space in which $X$ is a dense subspace. The well-known examples of proper compactifications of topological groups are the greatest ambit and the Roelcke compactification. Due to A.~Weil, a topological group $G$ is a dense subgroup of a compact topological group $K$ iff $G$ is precompact in the right uniformity. There are groups without (proper) semitopological semigroup compactifications~\cite{Megr2001}. Every topological group has a proper right topological semigroup compactification, for instance, its greatest ambit. A self-contained treatment of the theory of compact right topological semigroups and, in particular, of semigroup compactifications (non-proper) is in~\cite{BergJungMiln}. For additional information on proper compactifications of topological groups, see~\cite{8}.

In~\cite{KozlovSorin} it is noted that if $G$ is in the topology of pointwise convergence for the action $G\curvearrowright X$, provided $X$ is compact, then the Ellis enveloping semigroup is a proper $G$-compactification of $G$. This observation led to the notion of an {\it Ellis compactification} in~\cite[Definition 3.10]{KozlovLeiderman} or~\cite[Definition 2.1]{KozlovSorin}. An {\it Ellis compactification} $b G$ of a topological group $G$ is a proper $G$-compactification of $G$ such that $(bG, \bullet)$ is a right topological monoid and $\bullet|_{G\times b G}=\tilde\alpha$, where $\tilde\alpha$ is the extension of left action of $G$ on itself to $b G$ (Definition~\ref{defElliscomp}). 

In the present paper, we study  Ellis compactifications of topological groups. For this purpose, for a topological group $G$, we examine its {\it $\tau_p$-representation} in a topological space $X$ (if for the action  $G\curvearrowright X$ the topology of pointwise convergence on $G$ coincides with the original topology on $G$ and $X$ is a $G$-Tychonoff space, then we say that $G$ is {\it $\tau_p$-representable in $X$} (Definition~\ref{taurepres})). Further, we find $G$-compactifications of $X$ (examples illustrate that it is an easier problem than finding initially "good" compactifications of $G$). Finally, we obtain Ellis compactifications of $G$ (from $G$-compactifications of $X$) which are automatically right topological monoids. 

We use terminology from~\cite{KozlovLeiderman}, where the classification of proper semigroup compactifications is proposed. Besides Ellis compactifications of a group $G$, we examine sm-compactifications of $G$ (proper $G$-compactifications of $G$ which are semitopological monoids) and sim-compactifications of $G$ (proper $G$-compactifications of $G$ which are semitopological inverse monoids with continuous inverses). Every sim-compactification of $G$ is a sm-compactification of $G$, which in its turn is an Ellis compactification of $G$. If the poset of sm-compactifications of $G$ is non-empty, then the WAP-compactification of $G$ is its maximal element. 

Preliminary information about $G$-spaces, semigroups and uniformities on topological groups is contained in \S~\ref{prelim}. In Lemma~\ref{l1}, we justify the preservation of the topology of pointwise convergence when passing to $G$-compactifications of phase spaces. 

Our main result is the method of constructing Ellis compactifications described in \S~\ref{Elliscmp}. Proposition~\ref{mapelliscomp} is the starting point for the maps of Ellis compactifications and their comparisons. Proposition~\ref{Reesquot} details when one Ellis compactification is the Rees quotient of another. In Corollary~\ref{order} for the $\tau_p$-representation of a group $G$ in $X$ the map of $G$-compactifications of $X$ to Elis compactifications of $G$ is established. It preserves partial orders.

The uniformity on $G$ from a family of small subgroups (family of stabilizers of a finite number of points in the case of topology of pointwise convergence) is introduced in~\cite{Kozlov}. Its "spectral" character of definition turned out to be convenient for its comparison with the subspace uniformity on $G$ as the subset of Ellis compactification (Proposition~\ref{lemincl} and Corollary~\ref{coinunif1}).  In Proposition~\ref{a2-2} the sufficient condition for their equality is given. This equality yields Roelcke precompactness of $G$. 

In Proposition~\ref{autultr}, the order of Roelcke compactification and Ellis compactification of $G$ from its $\tau_p$-representation in its coset $X=G/H$ with respect to the neutral subgroup $H$ is established. The Ellis compactification $\mathfrak{E}_{\theta} (\beta_G X)$ of $G$ from its representation in $\beta_G X$  is less than or equal to  $\mathfrak{E} (b_r G)$, the least Ellis compactification of $G$ that is greater than or equal to the Roelcke compactification $b_r G$. Moreover, if  $\mathfrak{E}_{\theta} (\beta_G X)\geq b_r G$, then  $\mathfrak{E}_{\theta} (\beta_G X)=\mathfrak{E} (b_r G)$. 

Examples illustrate what Ellis compactifications can be obtained from $G$-compactifications of phase spaces. In \S~\ref{permutat} we examine an ultratransitive action of a group $G$ (in the topology of pointwise convergence) on a discrete space $X$. It is well-known that the permutation group ${\rm S}(\mathbb N)=({\rm S}(\mathbb N), \tau_p)$ of $\mathbb N$ is a  non-archimedean Roelcke precompact Polish group~\cite{Gau} (see also~\cite[Example 9.14]{RD}).  In~\cite[\S\ 12, Theorem 12.2]{GlasnerMegr2008} the description of Roelcke compactification $b_r  {\rm S}(\mathbb N)$ of ${\rm S}(\mathbb N)$ is given using its Ellis compactification from the $\tau_p$-representation of ${\rm S}(\mathbb N)$ in $\alpha\mathbb N$. It is a symmetric inverse semigroup $I_{\mathbb N}$ and $b_r {\rm S}(\mathbb N)$ is a WAP-compactification of ${\rm S}(\mathbb N)$. It is proved that $b_r  {\rm S}(\mathbb N)$ is homeomorphic to the Cantor set, and that ${\rm S}(\mathbb N)$ is not extremely amenable. In~\cite{BITsankov} it is proved that $b_r {\rm S}(\mathbb N)$ is a Hilbert compactification of ${\rm S}(\mathbb N)$ and all its factors are Hilbert--representable (see definitions in~\cite{GlasnerMegr}).

In Theorem~\ref{Roelcke precomp4-1} it is shown that for the permutation group ${\rm S}(X)=({\rm S}(X), \tau_p)$ of the infinite discrete space $X$  Alexandroff one-point compactification $\alpha X$ is the unique $G$-compactification of $G$-space $({\rm S}(X), X, \curvearrowright)$. ${\rm S}(X)$ is Roelcke precompact and its Roelcke compactification $b_r  {\rm S}(X)$ is Ellis compactification from its $\tau_p$-representation in $\alpha X$. Moreover, $b_r  {\rm S}(X)$ is a sim-compactification of ${\rm S}(X)$ (the symmetric inverse semigroup $I_{X}$) and a WAP-compactification of ${\rm S}(X)$. 
In Corollary~\ref{sisdiscr} it is shown that a countable number of sim-compactifications of ${\rm S}(X)$ exist and they are the Rees quotients of the Roelcke compactification $b_r  {\rm S}(X)$. 

In \S~\ref{chain} we examine an ultrahomogeneous discrete chain $X$ and LOTS $X$, and their automorphism group $\aut (X)$. It is well known that the group $\aut (\mathbb Q)$ (in the topology of pointwise convergence from its $\tau_p$-representation in a discrete chain $\mathbb Q$) is a non-archimedean Roelcke precompact Polish group, its Roelcke-compactification $b_r \aut (\mathbb Q)$ is not a compact semigroup. Further, the Hilbert compactification of $\aut (\mathbb Q)$ is a semitopological inverse monoid with continuous inverse, is a WAP-compactification of $\aut (\mathbb Q)$ and all its factors are Hilbert--representable~\cite {BITsankov}. $\aut (X)$ in the topology of pointwise convergence is extremely amenable for any ultrahomogeneous discrete chain $X$~\cite{pestov1998}. 

The group $\aut ([0, 1])$ (in the compact-open topology from its $\tau_p$-representation in LOTS $[0, 1]$) is Roelcke precompact, $b_r \aut ([0, 1])$ can be identified with the curves that connect the points $(0, 0)$ and $(1, 1)$ and monotone in the cardinal product order on $[0, 1]\times [0, 1]$~\cite{Usp} (see, for instance, \cite{GlasnerMegr2008}). A WAP-compactification of $\aut ([0, 1])$ is trivial~\cite{Megr2001} and $\aut ([0, 1])$ is extremely amenable~\cite{pestov1998}. 

We establish in Proposition~\ref{betaG} that the maximal $G$-compactification of an ultrahomogeneous discrete chain $X$ is the least linearly ordered compactification of a GO-space $X$, and the maximal $G$-compactification of a LOTS $X$ is the least linearly ordered compactification of $X$. In Lemmas~\ref{compLOTS} and~\ref{compd} all $G$-compactifications of $X$ are described. If $X$ is LOTS, there are two $G$-compactifications. If $X$ is a discrete chain, there are five $G$-compactifications. 

Corresponding Ellis compactifications $e_{c_m X}\aut (X)$ and $e_{c X}\aut (X)$ of $\aut (X)$, $X$ is LOTS, are described in Theorems~\ref{Roelcke precomp4-3-1} and~\ref{LOTcomp}. The group $\aut (X)$ is Roelcke precompact, $b_r \aut (X)<e_{c_m X}\aut (X)$,  $b_r \aut (X)$ and $e_{c X}\aut (X)$ are incomparable. They can't be sm-compactifications of $\aut (X)$ (Corollary~\ref{LOTSincom}). By the result of M.~Megrelishvili for the group $\aut ([0, 1])$ it seems that there are no sm-compactifications of $\aut (X)$.  If $c_m X=[0, 1]$, then $e_{c_m X} \aut (X)$ is a retract of the Helly space and is a first-countable, separable non-metrizable compact space (Remark~\ref{Helly}). 

Corresponding Ellis compactifications of $\aut (X)$, $X$ is a discrete chain, are described in Theorems~\ref{Roelcke precomp4-2-1}  and~\ref{allcompdiscr}. The group $\aut (X)$ is Roelcke precompact, $b_r \aut (X)\leq e_{b_m X}\aut (X)$, where $b_m X$ is the maximal $G$-compactification of $X$, and the equality is valid iff $X$ is a continuously ordered chain (Corollary~\ref{contchain}). Moreover, $e_{b_m X}\aut (X)$ is the least Ellis compactification of $\aut (X)$, that is greater than or equal to the Roelcke compactification $b_r \aut (X)$. $e_{b_m X}\aut (X)$ is not a sm-compactification of  $\aut (X)$ (Remark~\ref{sscompdiscr}).

The Ellis compactification $e_{\alpha X}\aut (X)$, where $\alpha X$ is the Alexandroff one-point compactification of $X$, is a sim-compactification of $\aut (X)$. It is isomorphic to the inverse monoid of partial automorphisms of $X$ and $e_{\alpha X}\aut (X)<b_r \aut (X)$ (Theorem~\ref{sischain}). Therefore, the WAP-compactification of $\aut (X)$ is a proper sm-compactification of  $\aut (X)$ and less than the Roelcke compactification $b_r \aut (X)$ (Remark~\ref{sscompdiscr}). In Corollary~\ref{sischaindiscr} it is shown that a countable number of sim-compactifications of $\aut (X)$ exist and they are the Rees quotients of the Ellis compactification $e_{\alpha X}\aut (X)$. If $X=\mathbb Q$, then all these sim-compactifications are metrizable and homeomorphic to the Cantor set.

The intermediate Ellis compactifications of $\aut (X)$ are not sm-compactifications of $\aut (X)$ (Remark~\ref{sscompdiscr}). If $X$ is not a continuously ordered chain, then the intermediate Ellis compactifications of $\aut (X)$ are incomparable with $b_r \aut (X)$ (Theorem~\ref{allcompdiscr}).

Ellis compactifications of automorphism groups of ultrahomogeneous cyclically ordered discrete sets (spaces) are established in~\cite{GSorinI} and~\cite{GSorin2}.

\medskip

We adopt terminology and notations from~\cite{Engelking} and~\cite{RD}. 

All spaces considered are Tychonoff, and we denote, if necessary, a (topological) space as $(X, \tau)$ where $\tau$ is a topology on the set $X$. We order topologies: $\tau'\geq\tau$ iff $\tau\subset\tau'$. An abbreviation ``nbd''  refers to an open neighbourhood of a point, nbd of $x$ is denoted $O_x$. The family of all nbds of the unit $e$ in the topological group $G$ is denoted $N_{G}(e)$. Subgroups of topological groups are closed subgroups. $|X|$ is the cardinality of the set (space) $X$.

Maps are continuous maps. The quotient onto map $f: X\to Y$ is an {\it elementary map} if the preimages of all points of $Y$, except maybe one, are one-point sets~\cite{Fed2003}.

A proper compactification of space $X$ is denoted $(bX, b)$ where $b X$ is a compact space, $b: X\to b X$ is a dense embedding. The order on compactifications: $(b'X, b')\geq (bX, b)$ iff $\exists\ \varphi: b' X\to b X$ such that $\varphi\circ b'=b$. $\varphi$ is the {\it map of compactifications}.

Uniform space is denoted as $(X, \mathcal U)$, $u$ is a uniform cover and ${\rm U}$ is the corresponding entourage of $\mathcal U$. Uniformity on a topological space is compatible with its topology. We write $\mathcal U\subset\mathcal V$ if the identity map  $\id: (X,\mathcal V)\to (X,\mathcal U)$ of uniform spaces is uniformly continuous. For ${\rm U}\in\mathcal U$ ${\rm U}\circ {\rm U}=\{(x, z)\ |\ \exists\ y\in X\ \mbox{such that}\ (x, y)\in {\rm U}, (y, z)\in {\rm U}\}$. For $x\in X$ and ${\rm U}\in\mathcal U$ the {\it ball} (with center $x$ and radius ${\rm U}$) is the set $B(x, {\rm U})=\{y\in X\ |\ (x, y)\in {\rm U}\}$. The star of a point $x$ with respect to the cover $u$ is denoted $\St (x, u)$. If ${\rm U\in \mathcal U}$ corresponds to a uniform cover $u\in \mathcal U$, then $B(x, {\rm U})=\St (x, u)$. If a cover $v$ is a refinement of a cover $u$, we use the notation $v\succ u$. For a cover $u$ of $Y$ and $X\subset Y$ $u\wedge X=\{U\cap X\ |\ U\in u\}$. If $X$ is a subset of a uniform space $(Y, \tilde{\mathcal U})$, then $\mathcal U=\{u\wedge X\ |\ u\in\tilde{\mathcal U}\}$ is the base of the {\it subspace uniformity} on $X$. The completion of $(X, \mathcal U)$ is denoted $(\tilde X, \tilde{\mathcal U})$. All the necessary (and additional) information about uniform structures (on groups) can be found in~\cite{Isbell} and~\cite{RD}.


\section{Preliminaries}\label{prelim} 

\subsection{$G$-spaces, admissible group topologies and topology of pointwise convergence}\label{admtop}

An action $\theta: G\times X\to X$ (designation $\theta (g, x)=gx$) of a group $G$ on a set $X$ ({\it phase space}) is called {\it effective} if the {\it kernel of the action} $\{g\in G\ |\ gx=x,\ \forall\ x\in X\}$ is the unit $e$ of $G$. If $G$ effectively acts on $X$, then $G\subset {\rm S}(X)$, where ${\rm S}(X)$ is the {\it permutation group} of $X$. In the paper, only effective actions are examined.

The family of finite (nonempty) subsets of a set (or space) $X$ is denoted $\Sigma_X$ and directed by inclusion $\sigma\leq\sigma'$ iff $\sigma\subset\sigma'$. The subgroup $\st_{\sigma}=\{g\in G\ |\ gx=x,\ x\in\sigma\in\Sigma_X\}=\bigcap\limits_{x\in\sigma}\st_{x}$ of $G$ is a {\it stabilizer}  (or stabilizer of points from $\sigma$).

If $T\subset G$, $Y\subset X$, then $TY=\theta (T, Y)=\bigcup\{ty\ |\ t\in T,\ y\in Y\}$. For $x\in X$ $\theta_x: G\to X$, $\theta_x(g)=\theta (g, x)$, is an {\it orbit map}. 

\medskip

Under a continuous action $\theta: G\times X\to X$ of a topological group $G$ on a space $X$, the triple $(G, X, \theta)$ is called a {\it $G$-space} (abbreviation $X$ is a $G$-space). A  continuous action $\theta: G\times X\to X$ will be denoted $G\curvearrowright X$ if we don't need the special designation for the map. If $X$ is a compact space, then the $G$-space $(G, X, \theta)$ will be called a {\it compact $G$-space} (or, simply, $X$ is a compact $G$-space).

The map $f: X\to Y$ of $G$-spaces $(G, X, \theta_X)$ and  $(G, Y, \theta_Y)$ is a {\it $G$-map} (or {\it equivariant map}) if $f(\theta_X (g, x))=\theta_Y(g, f(x))$, $g\in G$, $x\in X$. 

A uniformity $\mathcal U_X$ on a phase space $X$ of a $G$-space $(G, X, \theta)$ is called an {\it equiuniformity}~\cite{Megr1984} (see, also~\cite{Megr0}) if the action $G\curvearrowright X$ is {\it saturated} (any homeomorphism from $G$ is uniformly continuous with respect to $\mathcal U_X$) and is {\it bounded} (for any $u\in\mathcal U_X$ there exist $O\in N_G(e)$ and $v\in\mathcal U_X$ such that the cover $Ov=\{OV\ |\ V\in v\}$ is refined in $u$). In this case  $(G, X, \theta)$ a {\it $G$-Tychonoff space} (or, simply, $X$ is a $G$-Tychonoff space), the action is extended to the continuous action $\tilde\theta: G\times \tilde X\to\tilde X$ on the completion $(\tilde X, \tilde{\mathcal U}_X)$ of $(X, \mathcal U_X)$. The extension $\tilde{\mathcal U}_X$ of $\mathcal U_X$ is an equiuniformity on $\tilde X$. The embedding $\jmath: X\to\tilde X$ is a $G$-map, the pair $(\tilde X, \jmath)$  is a  {\it $G$-extension of $X$} and $\jmath (X)$ is a dense {\it invariant} subspace of $\tilde X$. $G$-spaces as extensions will also be denoted as $(G, (\tilde X, \jmath), \tilde\theta)$.

\medskip

If $\mathcal U_X$ is a totally bounded equiuniformity on $X$, then $(\tilde X, j)$ is a  {\it $G$-compactification} (or {\it equivariant compactification}) of $G$-space $X$. If  $(G, X, \curvearrowright)$ is a $G$-Tychonoff space,  then the maximal (totally bounded) equiuniformity  $p\mathcal U_X^{max}$ on $X$ exists and {\it the maximal $G$-compactification}  $(\beta_G X, \jmath_{\beta})$ is the completion of $(X, p\mathcal U_X^{max})$. $p\mathcal U_X^{max}$ is the {\it precompact replica}~\cite{Isbell} of the maximal equiuniformity  $\mathcal U_X^{max}$ on $X$ and $\beta_G X$ is the {\it Samuel compactification of $(X, \mathcal U_X^{max})$} (see~\cite[Ch.\ 8, Problem 8.5.7]{Engelking}).

The maximal equiuniformity $\mathcal U_{G/H}$ on a coset space $G/H$ of a topological group $G$ ($H$ is a closed subgroup of $G$, action by multiplication on the left) is described in~\cite{ChK} or~\cite{ChK2}. There such actions are called {\it open} and the covers of the base of the maximal equiuniformity are of the form 
$$\{Ox\ |\ x\in G/H\},\ O\in N_G(e).$$

\medskip

If $X$ is a topological space, $\Hom(X)$ is the group of its homeomorphisms, then a group $G$ effectively acts on $X$ if $G\subset\Hom(X)$. If $X$ is a discrete space, then $\Hom(X)={\rm S}(X)$.

A topology in which a group $G$ is a topological group and its (effective) action $G\curvearrowright X$ is continuous is called an {\it admissible group topology}~\cite{Arens}. For an action of $G$ on a compact space $X$ the {\it compact-open topology} $\tau_{co}$ is the least admissible group topology on $G$~\cite{Arens}. 

If the {\it topology of pointwise convergence} $\tau_p$ (the subbase is formed by sets of the form $[x, O]=\{g\in G\ |\ gx\in O\}$, $x\in X$, $O$ is open in $X$) is an admissible group topology on $G$ for its action $G\curvearrowright X$, then it is the least admissible group topology \cite[Lemma 3.1]{Kozlov}. 

For a uniform space $(X, \mathcal L)$ an action $\theta: G\times (X, \mathcal L)\to  (X, \mathcal L)$ is {\it uniformly equicontinuous} if $\{\theta^g: X\to X,\ \theta^g(x)=\theta (g, x)\ |\ g\in G\}$  is a uniformly equicontinuous family of maps on $X$. Equivalently, for any $u\in\mathcal L$ there exists $v\in\mathcal L$  such that $gv=\{gV\ |\ V\in v\}\succ u$, for any $g\in G$. 
If an action  $G\curvearrowright (X, \mathcal L)$ is uniformly equicontinuous, then the topology of pointwise convergence $\tau_p$ is the least admissible group topology on $G$~\cite[Lemma 3.1]{Kozlov}. Moreover, $((G, \tau_p), X, \theta)$ is a $G$-Tychonoff space, see, for instance, \cite{Megr1984}. 

For an action of $G$ on a discrete space $X$, the topology of pointwise convergence $\tau_p$ (also called {\it permutation topology} $\tau_{\partial}$ in this case, and the subbase is formed by sets of the form $[x, y]=\{g\in G\ |\ gx=y \}$, $x, y\in X$) is an admissible group topology. $(G, \tau_{\partial})$ is {\it non-Archimedean} (a unit nbd base is formed by clopen subgroups (stabilizers)) and the $G$-space $((G, \tau_{\partial}), X, \curvearrowright)$ is $G$-Tychonoff.

The extremal properties of topologies $\tau_p$ and $\tau_{co}$ yield that if $\tau_p$ is an admissible group topology on $G$ acting on a compact space $X$, then $\tau_p=\tau_{co}$.

\begin{rem}{\rm 
Example from~{\rm\cite{Megr1988}} shows that $\tau_p$ may be an admissible group topology on a group $G$ of homeomorphisms of $X$ acting effectively, but the $G$-space $(G=(G, \tau_p), X,  \curvearrowright)$ may not be $G$-Tychonoff.}
\end{rem}

\begin{lem}\label{l1} Let $(G=(G, \tau_p), X, \theta)$ be $G$-Tychonoff and $b X$ is a $G$-compactification of $X$. Then the topology of pointwise convergence $\tau_p^{b X}$ for the extended action $\tilde\theta: G\times  b X\to b X$ is the least admissible group topology on $G$ and $\tau_p^{b X}=\tau_p$.
\end{lem}

\begin{proof} The compact-open topology $\tau_{co}$ induced by the extended action $\tilde\theta$ is the least  admissible group topology on $G$ and $\tau_p$ is an admissible group topology on $G$. Therefore, $\tau_p\geq\tau_{co}$. Evidently, $\tau_{co}\geq\tau_p^{b X}\geq\tau_p$. Hence, $\tau_p^{b X}=\tau_p$.
\end{proof}

\medskip

Let $G=(G, \tau)$ be a topological group. The $G$-space $(G, G, \alpha)$  with action $\alpha$ by multiplication on the left is $G$-Tychonoff and $\tau=\tau_p$. Indeed, $\tau_p$ is an admissible group topology since the action is uniformly equicontinuous with respect to the left uniformity $L$ on $G$~\cite[Ch. X, § 3, item 5]{Burb} and $\tau\geq\tau_p$. If $O\in N_{(G, \tau)}(e)$, then the set $[e, O]\in\tau_p$ and $[e, O]=O$. Thus,  $\tau=\tau_p$. If $b G$ is a $G$-compactification of $G$-space $(G, G, \alpha)$, then $\tilde\alpha: G\times b G\to b G$ is the extension of $\alpha$. 


\subsection{Uniformities on topological groups}\label{unifgr}

On a topological group $G$ four uniformities are well-known. The {\it right uniformity} $R$ (the base is formed by the covers $\{Og=\bigcup\{hg\ |\ h\in O\}\ |\ g\in G\}$, $O\in N_G(e)$), the {\it left uniformity} $L$ (the base is formed by the covers $\{gO=\bigcup\{gh\ |\ h\in O\}\ |\ g\in G\}$, $O\in N_G(e)$), the {\it two sided uniformity} $R\vee L$ (the least upper bound of the right and left uniformities) and the {\it Roelcke uniformity} $L\wedge R$ (the greatest lower bound of the right and left uniformities) (the base is formed by the covers $\{OgO=\bigcup\{hgh'\ |\ h, h'\in O\}\ |\ g\in G\}$, $O\in N_G(e)$). A group  $G$ is {\it Roelcke precompact} if Roelcke uniformity is totally bounded. All the necessary information about these uniformities can be found in~\cite{RD}. 

{\it Roelcke compactification} $b_r G$ of $G$ is the Samuel compactification of $(G, L\wedge R)$, i.e. the completion of $(G, (L\wedge R)_{fin})$, where $(L\wedge R)_{fin}$ is the {\it precompact replica} of $L\wedge R$. If $G$ is Roelcke precompact, then $(L\wedge R)_{fin}=L\wedge R$ and {\it Roelcke compactification} $b_r G$ is the completion of $(G, L\wedge R)$.

For every equiuniformity $\mathcal U$ on $G$ of a $G$-space $(G, G, \alpha)$ one has $\mathcal U\subset R$, $L\wedge R$  is an equiuniformity and $R$ is the maximal equiunuformity.

\medskip

{\it Uniformity on $G$ from a family of small subgroups is introduced in~{\rm\cite[\S\ 4]{Kozlov}}.}
Let  $(G=(G, \tau_p), X,  \curvearrowright)$ be a $G$-space.
The family of stabilizers  
$$\st_{\sigma}=\bigcap\{\st_{x}\ |\ x\in\sigma\},\ \sigma\in\Sigma_X,$$ 
is a {\it directed family  of small} (closed) {\it subgroups of $G$} (i.e. $\forall\ O\in N_G(e)$ $\exists\ \sigma\in\Sigma_X$ such that $\st_{\sigma}\subset O$; $\st_{\sigma}\leq\st_{\sigma'}\Longleftrightarrow \sigma\leq\sigma'$)  which defines the uniformity $R_{\Sigma_X}$ on $G$. The base of  $R_{\Sigma_X}$ is formed by the covers 
$$\{Og\st_{\sigma}=\bigcup\{hgh'\ |\ h\in O,\ h'\in\st_{\sigma}\}\ |\ g\in G\},\ O\in N_G(e),\  \sigma\in\Sigma_X.$$  
Uniformity  $R_{\Sigma_X}$ is an equiuniformity (for the action $\alpha$) and has a "spectral" representation. It is initial~\cite[Proposition-Definition 0.16]{RD} with respect to the quotient maps $q_{\sigma}: G\to (G/\st_{\sigma}, \mathcal U_{G/\st_{\sigma}})$  (cosets $G/\st_{\sigma}$ with the maximal equiuniformities $\mathcal U_{G/\st_{\sigma}}$ (the quotient uniformities of the right uniformity $R$ on $G$)~\cite[Theorem 5.21]{RD} and~\cite{ChK}), $\sigma\in\Sigma_X$.
 
\begin{rem}{\rm\cite[Corollary 4.5]{Kozlov}}\label{a2-1}
{\rm (A) $L\wedge R\subset R_{\Sigma_X}\subset R$. Hence, if the uniformity $R_{\Sigma_X}$ is totally bounded, then $G$ is Roelcke precompact. Other results that guarantee the Roelcke precompactness of an acting group $G$ are the following. 

(a) If a (closed) subgroup $H$ of a topological group $G$ is Roelcke precompact and the maximal equiuniformity on the coset space $G/H$ is totally bounded, then $G$ is Roelcke precompact~\cite[Proposition 6.4]{Megr}.

(b) If a topological group $G$ has a directed family  $\mathcal K$ of small subgroups such that the maximal equiuniformity on a coset space $G/H$ is totally bounded, $H\in\mathcal K$, then $G$ is Roelcke precompact~\cite[Corollary 4.5]{Kozlov}.

(c) If the action  $G\curvearrowright X$, where $X$ is a discrete space and group $G$ is in the topology of pointwise convergence, is oligomorphic (equivalently, the maximal equiuniformity on $X$ is totally bounded~\cite[Theorem 3.1]{Sorin2025}),  then $G$ is Roelcke precompact~\cite {Tsan}.

\medskip

\rm (B) For a group $G$ {\rm(}acting on a discrete space $X${\rm)} in the topology of pointwise convergence the stabilizers $\st_{\sigma}$, $\sigma\in\Sigma_X$, are clopen subgroups {\rm(}nbds of the unit{\rm)} and $L\wedge R=R_{\Sigma_X}$. Indeed, the base of the Roelcke uniformity $L\wedge R$ on $G$ in this case is formed by the covers
$$\{\st_{\sigma}g\st_{\sigma}\ |\ g\in G\},\ \sigma\in\Sigma_X,$$
which coincides with the base of $R_{\Sigma_X}$.}
\end{rem}


\subsection{Semigroups}\label{semgr}
A {\it semigroup} $(S, \bullet)$ is a set $S$ with an associative internal binary operation (multiplication $\bullet$). Monoid $M$ is a semigroup with identity $e$: $\forall\ x\in M\ (e\bullet x=x\bullet e=x)$. Every group is a monoid.  

A semigroup (monoid) $S$ is an {\it inverse semigroup} ({\it inverse monoid}) if 
$$\forall\ a\in S\  \exists\ \mbox{the unique}\ a^*\in S\  \mbox{such that}\ a\bullet a^*\bullet a=a\ \&\ a^*\bullet a\bullet a^*=a^*,$$
$a^*$ is an inverse of $a$. Every inverse semigroup (monoid) is a semigroup (monoid) with involution $s: a\to a^*$.

The {\it symmetric inverse semigroup} $I_X$ on $X$ is the collection of all partial bijections between subsets of $X$ with composition as the binary operation (see, for example,~\cite{CliffPrest} or~\cite{Lawson}). The symmetric inverse semigroup is an inverse monoid. 

A map $s: (S, \bullet)\to (S', \bullet')$ of semigroups is a {\it homomorphism} if $s(x\bullet y)=s(x)\bullet' s(y)$ for all $x, y\in S$. If $G$ is a group and $s: G\to (S, \bullet)$ is a homomorphism, then $s(G)$ is a group and a subsemigroup of $S$. See,  for instance, \cite{CliffPrest} and~\cite{Lawson} for additional information. 

\medskip

A  {\it semitopological} ({\it right topological}) {\it semigroup} ({\it monoid}) is a semigroup  (monoid) $(S, \bullet)$ on a topological space $S$ with separately continuous multiplication  $\bullet$ (multiplication $\bullet$ is continuous on the right (i.e. $\forall\ x\in S$ the map $\bullet_x: S\to S$, $\bullet_x(y)=y\bullet x$ is continuous)). A  {\it semitopological inverse semigroup} ({\it inverse monoid}) is an inverse semigroup (inverse monoid)  $(S, \bullet)$ on a topological space $S$ with separately continuous multiplication $\bullet$. All the necessary (and additional) information about semigroups which are topological spaces can be found, for instance, in~\cite{BergJungMiln} and~\cite{ArhTk}.

\medskip

$I$ is an {\it ideal} in the semigroup $(S, \bullet)$ if $x\bullet I=\{x\bullet y\ |\ y\in I\}\subset I$ and  $I\bullet x\subset I$ for all $x\in S$. Let $I$ be an ideal in the semigroup $(S, \bullet)$ with $0$  (if there is no $0$ in $S$ it can be added with the natural extension of multiplication $0\bullet x=x\bullet 0=0$, $x\in S$,  and involution $s(0)=0$ if $S$ is a semigroup with involution). Define a relation $\rho_I$ on $S$ 
$$(x, y)\in\rho_I\ \Longleftrightarrow\ \mbox{either}\ x, y\in I\ \mbox{or}\ x=y.$$
Then $\rho_I$ is a {\it congruence} on $S$. The quotient $S/\rho_I$ is isomorphic to the set $(S\setminus I)\cup\{0\}$ equipped with the following multiplication $\bullet_I$: if $x, y\in S\setminus I$, then $x\bullet_I y=x\bullet y$ if $x\bullet y\in S\setminus I$ and   $x\bullet_I y=0$ in all other cases. $(S/\rho_I, \bullet_I)$ is a semigroup with $0$. Such quotients are called {\it Rees quotients} (see, for instance, \cite{Lawson}).

If $s$ is an involution on $(S, \bullet)$, $I$ is an ideal, then $s(I)$ is an ideal ($\forall\ x\in S$ $s(I)\bullet x=s(s(x)\bullet I)\subset s(I)$ and $x\bullet s(I)=s(I\bullet s(x))\subset s(I)$). If $s(I)=I$, then the map $s_I: S/\rho_I\to S/\rho_I$, $s_I(x)=s(x)$ if $x\not\in I$, $s_I(0)=0$ is an involution on $S/\rho_I$. Indeed, $s_I\circ s_I$ is the identity map on $ S/\rho_I$, and 
$$s_I(x\bullet y)=\left\{\begin{array}{ccc}
0=s_I(y)\bullet s_I(x) & \mbox{if} &  x\bullet y=0, \\
s(x\bullet y)=s(y)\bullet s(x)=s_I(y)\bullet s_I(x) & \mbox{if} &  x\bullet y\ne 0. \\
\end{array}
\right.$$

If $(S, \bullet)$ is an inverse semigroup, $I$ is an ideal, then $I^*=I$ ($I$ is invariant under the involution defined by the inverse). Indeed, since $x^*\bullet x\bullet x^*=x^*$, if $x\in I$, then $x^*\bullet x\bullet x^*\in I$ and, hence, $x^*\in I$, $I^*\subset I$. The opposite inclusion is proved similarly. 

\begin{lem}\label{ideals}
{\rm(a)} If $(S, \bullet)$ is a right topological semigroup {\rm(}monoid{\rm)}, $I$ is a closed in $S$ ideal, then $(S/\rho_I, \bullet_I)$ in the quotient topology is a right topological semigroup {\rm(}monoid{\rm)} with $0$.

{\rm(b)} If $(S, \bullet)$ is a semitopological semigroup {\rm(}monoid{\rm)}, $I$ is a closed in $S$ ideal, then $(S/\rho_I, \bullet_I)$ in the quotient topology is a semitopological semigroup {\rm(}monoid{\rm)} with $0$. If, additionally, $s$ is a {\rm(}continuous{\rm)} involution on $(S, \bullet)$ and $s(I)=I$, then $s_I$ is a {\rm(}continuous{\rm)} involution on $S/\rho_I$.

{\rm(c)} If $(S, \bullet)$ is a semitopological inverse semigroup {\rm(}monoid{\rm)} with continuous inverse, $I$ is a closed in $S$ ideal, then $(S/\rho_I, \bullet_I)$  in the quotient topology is a semitopological inverse semigroup {\rm(}monoid{\rm)} with $0$. 
\end{lem}

\begin{proof}
Let $q: S\to S/\rho_I$ be the quotient map (homomorphism of semigroups). Then $q\circ\bullet=\bullet_I\circ (q\times q)$.
$q(I)$ is $0$ in $S/\rho_I$. If $S$ is a monoid with identity $e$, then $q(e)$ is the identity in $S/\rho_I$. 

\medskip

The restriction $q|_{S\setminus I}$ is a bijection of $S\setminus I$ onto $S/\rho_I\setminus\{0\}$. Therefore, if $O$ is open in $S$ and $O\cap I=\emptyset$ or $I\subset O$, then $q(O)$ is open in $S/\rho_I$. 

$\forall\ y\in S/\rho_I$ one has $y\bullet_I 0=0$ and the constant map $(\bullet_I )_0(y)=y\bullet_I 0$ is continuous. 

For $q(x)\in S/\rho_I$, $x\ne I$, 
$$(\bullet_I )_{q(x)} (q(y))=\left\{\begin{array}{ccc}
0 & \mbox{if} &  y\bullet x\in I, \\
q(\bullet_x (y)) & \mbox{if} &  y\bullet x\not\in I. \\
\end{array}
\right.$$
If $(\bullet_I )_{q(x)} (q(y))=0$, then arbitrary nbd of $0$ can be assumed to be $q(O)$ where $O$ is open in $S$, $I\subset O$. Since $(\bullet_I )_{q(x)} (q(y))=q(\bullet_x (y))=0$, if $y\not\in I$ there is its nbd $W_y$ such that $W_y\cap I=\emptyset$, $\bullet_x (W_y)\subset O$. Then $q(W_y)$ is a nbd of $q(y)$ and $(\bullet_I )_{q(x)}(q(W_y))\subset q(O)$. 

If $y\in I$, then $\forall\ z\in I$ $\exists$ its nbd $W_z$ such that $\bullet_x (W_z)\subset O$. The union $W$ of a family of nbds $W_{z}$, $z\in I$, contains $I$. $W$ is a nbd of $y$ and $\bullet_x (W)\subset O$. Hence, $q(W)$ is a nbd of $q(y)=0$ and $(\bullet_I )_{q(x)} (q(W))\subset q(O)$.

If $(\bullet_I )_{q(x)} (q(y))\ne 0$, then $y\not\in I$ and one can use the above arguments. The maps $(\bullet_I )_{q(x)}$ are continuous for all $q(x)\in S/\rho_I$ and $S/\rho_I$ is a right topological semigroup. 

\medskip

The same arguments shows that if $(S, \bullet)$ is a semitopological semigroup, then $(S/\rho_I, \bullet_I)$ is a semitopological semigroup. 

\medskip

If $s$ is an involution on $(S, \bullet)$, then the involution $s_I$ is continuous. Indeed, $s_I(0=q(I))=0$, $s_I(q(x))=q(s(x))$, $x\not\in I$. An arbitrary nbd of $0$ can be $q(O)$ where $O$ is open in $S$, $I\subset O$. Since $s$ is continuous, one can find a nbd $W$ of $I$ such that $s(W)\subset O$ (as above). Hence, $q(W)$ is a nbd of $0$ and $s_I (q(W))\subset q(O)$.

If $x\not\in I$, then $\forall$ nbd $O_{s(x)}$ of $s(x)$, $O_{s(x)}\cap I=\emptyset$, $\exists$ nbd $W_x$ such that $W_x\cap I=\emptyset$ and $s(W_x)\subset O$. Therefore, $s_I (q(W_x))\subset O_{s(x)}$.

\medskip

The last statement follows from the previous ones and the fact that the Rees quotient of an inverse semigroup is an inverse semigroup. 
\end{proof}

\begin{rem}
{\rm The map $q: S\to S/\rho_I$ is an elementary map.}
\end{rem}


\section{Ellis compactification of a group from its $\tau_p$-representation in a compact space}\label{Elliscmp}

\begin{df}~\cite[Definition 3.10]{KozlovLeiderman} {\rm(}see, also, {\rm\cite[Definition 2.1]{KozlovSorin})}\label{defElliscomp}
Let $G$ be a topological group.  An Ellis compactification $b G$ of $G$ is a $G$-compactification of $G$ such that $(bG, \bullet)$ is a right topological monoid and $\bullet|_{G\times b G}=\tilde\alpha$, where $\tilde\alpha$ is the extension of the action $\alpha$ of $G$ on itself to $b G$. 
\end{df}

$\mathbb{E} (G)$ is the poset of Ellis compactifications of $G$.

\medskip

Since the Cartesian product $X^X$ may be regarded as the set of self-maps of $X$ ($f=(f(x))_{x\in X}\in X^X$), $X^X$ is a monoid (mu{\rm l}tiplication is composition $\circ$ of maps, the identity map of $X$ is the identity in $X^X$).

If $X$ is a topological space, then $(X^X, \circ)$ is a right topological monoid (multiplication on $X^X$ in the Tychonoff topology is continuous on the right).   

\begin{pro}\cite[\S\ 2.7 Constructions,\ 2]{Vries}
Let $(G, X, \theta)$ be a $G$-space. Then $(G, X^X, \theta_{\Delta X})$ is a $G$-space, where 
$$\theta_{\Delta X} (g,  f)(x)=\theta (g, f(x)),\ x\in X,$$
is the diagonal action {\rm(}induced by $\theta${\rm)}.
\end{pro}

\begin{pro}~\cite[Proposition 4.2]{KozlovLeiderman}  {\rm(}see, also, {\rm\cite[Proposition 2.6, Theorem 2.7]{KozlovSorin})}\label{diagactalg}
Let $(G=(G, \tau_p), X, \theta)$ be a $G$-space. Then 
\begin{itemize}
\item[(1)] the map
$$\imath_X: G\to X^X,\ \imath_X (g)(x)=\theta (g, x),\ x\in X,$$
is a topological isomorphism of $G$ onto the subsemigroup of $X^X$.
\item[(2)]  $\theta_{\Delta X} (g,  f)=\imath_X (g)\circ f$,
\item[(3)]  $\imath_X$ is a $G$-map of $G$-spaces $(G, G, \alpha)$ and $(G, X^X, \theta_{\Delta X})$, the following diagram is commutative 
$$\begin{array}{ccc}
\quad G\times G & \stackrel{{\id\times \imath_X}}\hookrightarrow & G\times X^X \\
 \alpha  \downarrow &   & \quad \downarrow \theta_{\Delta X}\\
\quad G & \stackrel{\imath_X}\hookrightarrow & X^X.
\end{array}$$
\end{itemize}
\end{pro}

Since $\imath_X (G)$ is an invariant subset of the $G$-space $(G, X^X, \theta_{\Delta X})$, $\cl (\imath_X (G))$ is an invariant subset and a $G$-compactification of $G$ if $X$ is compact. $\cl (\imath_X (G))$ is a right topological monoid~\cite{Ellis}. The restriction of $\circ$ to $\imath_X (G)\times \cl (\imath_X (G))$ coincides with the extended action $\tilde\alpha: G\times \cl (\imath_X (G))\to \cl (\imath_X (G))$, $\tilde\alpha=\theta_{\Delta X}|_{G\times\cl (\imath_X (G))}$. This follows from the commutativity of the diagram 
$$\begin{array}{ccccc}
\quad G\times G & \stackrel{\id\times \imath_X}\hookrightarrow & G\times \cl (\imath_X (G)) &  \stackrel{\imath_X\times\id}\hookrightarrow    & \cl (\imath_X (G))\times \cl (\imath_X (G))\\
\alpha  \downarrow &   & \quad \downarrow \tilde\alpha  &  &  \quad \downarrow \circ \\
\quad G & \stackrel{\imath_X}\hookrightarrow & \cl (\imath_X (G)) &  \stackrel{\id} {\longrightarrow} & \cl (\imath_X (G)).
\end{array}$$
Hence,  $\cl (\imath_X (G))$ is an Ellis compactification of $G$. Notation $e_X G=\cl (\imath_X (G))$. 


\subsection{Maps of Ellis compactifications}

\begin{df}\label{taurepres}
A topological group $G=(G, \tau)$ is {\it $\tau_p$-representable in the {\rm(}Tychonoff{\rm)} space $X$} if there exists a continuous {\rm(}effective{\rm)} action $\theta: G\times X\to X$, $((G, \tau_p), X, \theta)$ is a $G$-Tychonoff space and $\tau_p=\tau$.
\end{df}

Let a topological group $G$ be $\tau_p$-representable in compact spaces $X$ and $Y$ ($(G, X, \theta_X)$ and $(G, Y, \theta_Y)$ are correspondent compact $G$-spaces), $\varphi: X\to Y$ is an onto (perfect) $G$-map and $\sim_{\varphi}$ is the equivalence relation on $X$ induced by $\varphi$
$$x\sim_{\varphi}y\ \Longleftrightarrow\ \varphi (x)=\varphi (y).$$
By $[x]$ (a compact subset), the equivalence class of $x\in X$ is denoted. 

Let $(e_X G, \imath_X)$, $(e_Y G, \imath_Y)$ be Ellis compactifications of $G$.  

\begin{lem}\label{mapeqc}
$\forall\ f\in e_X G$ and $\forall\ x\in X$ $f([x])\subset [f(x)]$.
\end{lem}

\begin{proof}
Let $x\sim_{\varphi}y$. $\forall\ g\in G$ $\theta_X (g, x)\sim_{\varphi}\theta_X (g, y)$. Hence, $g([x])\subset [g(x)]$, $g\in G$. 

Take $f\in e_X G$ and assume that $\exists\ x, y\in X$ such that $x\sim_{\varphi} y$ but $[f(x)]\cap [f(y)]=\emptyset$. 
Let $U$ and $V$ be nbds of $ [f(x)]$ and $[f(y)]$ respectively, such that $\varphi (U)\cap\varphi (V)=\emptyset$. Then $f\in ([x, U]\cap [y, V])$. However, $([x, U]\cap [y, V])\cap\imath_X (G)=\emptyset$ and, therefore, $f\not\in  e_X G$. The obtained contradiction finishes the proof. 
\end{proof}

The Cartesian product $\varphi^X:   X^X\to  Y^X$ of maps $\varphi$ is correctly defined.  

Let a subset $Z\subset X$ consist of one representative from each equivalence class of $\sim_{\varphi}$. $\varphi_{Z}: Y^X\to  Y^Y$ is the composition of projection $\pr_{Y}: Y^X\to  Y^Z$ and the map $Y^Z\to Y^Y$ which identifies coordinates under the bijection $\varphi|_Z$. The composition $\Phi=\varphi_{Z}\circ\varphi^X: X^X\to Y^Y$ is correctly defined.  

\begin{pro}\label{mapelliscomp}
$\Phi|_{e_X G}: e_X G\to e_Y G$ is the $G$-map of compactifications of $G$ and homomorphism of monoids. 
\end{pro}

\begin{proof}
The following diagram 
$$\begin{array}{ccl}
X^X & \stackrel{\Phi}{\longrightarrow} &  Y^Y \\
\imath_X \nwarrow &   &  \nearrow \imath_Y \\
 & G & 
\end{array}$$
is commutative.  In fact, for any  $y\in Y$, $g\in G$,  
$$\imath_Y (g)(y)=\theta_Y (g, y)\ \mbox{and}$$ 
$$[(\Phi\circ\imath_X) (g)] (y)=\pr_{y}\big((\Phi\circ\imath_X) (g)\big)=\pr_{y}\big((\varphi_{Z}\circ\varphi^X\circ \imath_X)(g)\big)=$$
\centerline{($\pr_{y}: Y^Y\to\{y\}\times Y=Y$ is a projection, $\pr_{y}\circ\varphi_{Z}=\pr_{x}$, $x\in Z$, $\varphi (x)=y$)} 
$$=\pr_{x}\big((\varphi^{X}\circ \imath_X)(g)\big)=\varphi\big(\imath_X (g)(x)\big)=\varphi\big(\theta_X (g, x)\big)=$$
\centerline{($\varphi$ is a $G$-map)}
$$=\theta_Y (g, \varphi (x))=\theta_Y (g, y).$$
Since $\imath_X$ and $ \imath_Y$ are embeddings, $\Phi|_{e_X G}: e_X G\to e_Y G$ is a map of compactifications. By~\cite[Proposition 3.26]{KozlovLeiderman} it is a $G$-map and a homomorphism of monoinds.
\end{proof}

The map $\Phi$ defines a congruence $\sim_{\Phi}$ on a semigroup $e_X G$
$$f\sim_{\Phi}h\ \Longleftrightarrow\ \Phi (f)=\Phi (h).$$
By $[f]$ the equivalence class of $f\in  e_X G$ is denoted.

\begin{lem}
$f\sim_{\Phi}h\ \Longleftrightarrow\ f(x)\sim_{\varphi} h(x),\ \forall\ x\in X$.
\end{lem}

\begin{proof}
$$f\sim_{\Phi}h\ \Longleftrightarrow\ \Phi (f)=\Phi (h)\  \Longleftrightarrow\ ((\varphi_Z\circ \varphi^X) (f))(y)=((\varphi_Z\circ \varphi^X) (h))(y),\ \forall\ y\in Y, \Longleftrightarrow$$
$$\pr_y((\varphi_Z\circ \varphi^X) (f))=\pr_y((\varphi_Z\circ \varphi^X) (h)),\ \forall\ y\in Y, \Longleftrightarrow$$
\centerline{($\varphi (x)=y$)} 
$$\pr_x(\varphi^X (f))=\pr_x(\varphi^X (h)),\ \forall\ x\in Z,\Longleftrightarrow \varphi(f (x))=\varphi (h(x)),\ x\in Z, \Longleftrightarrow$$
$$ f(x)\sim_{\varphi} h(x),\ \forall\ x\in Z, \Longleftrightarrow$$
\centerline{(by Lemma~\ref{mapeqc})} 
$$ f(x)\sim_{\varphi} h(x),\ \forall\ x\in X.$$
\end{proof}

Let $\varphi: X\to Y$ be an (onto) elementary $G$-map ($\varphi|_{X\setminus\varphi^{-1}(y)}$, $y\in Y$, is a homeomorphism). Then $\st_y=G$ ($y$ is a fixed point in $Y$). Put $I=\{f\in e_X G\ |\ f(x)\in\varphi^{-1}(y),\ x\in X\}$, $I$ is a compact subset of $e_X G$. 

\begin{pro}\label{Reesquot}
If $I\ne\emptyset$, then $I$ is a compact ideal of $e_X G$. 

If, additionally, the condition 
$$\forall\ f, h\in e_X G\ \mbox{such that}\ f|_{X\setminus\varphi^{-1}(y)}=h|_{X\setminus\varphi^{-1}(y)}\ne\emptyset\ \Longrightarrow\ f=h\eqno{\rm (EF)}$$
is valid, then the Rees quotient $e_X G/I$ is topologically isomorphic to $e_Y G$.
\end{pro}

\begin{proof}
Since $\varphi$ is a $G$-map, $\varphi^{-1}(y)$ is an invariant subset in $X$. Therefore, $f(\varphi^{-1}(y))\subset\varphi^{-1}(y)$, $\forall\ f\in e_X G$.

From the definition of $I$ it follows that $\forall\ h\in I$ and $\forall\ f\in e_X G$
$$f\circ h\ \mbox{and}\ h\circ f\in I.\ \mbox{Hence},\ I\ \mbox{is an ideal}.$$

From (EF) and definition of $\Phi: e_X G\to  e_Y G$, it follows that $\Phi|_{e_X G\setminus I}$ is the identity map and $\Phi (I)$ is a constant map $c_{y}: Y\to y$. The quotient space $e_X G/I$ of $e_X G$ ($q: e_X G\to e_X G/I$ is the quotient map, $q|_{e_X G\setminus I}$ is the identity map and $q (I)$ is a one-point set) is defined by Lemma~\ref{ideals}. The continuous map $p: e_X G/I\to e_Y G$ such that $\Phi=p\circ q$ is defined. It is a homeomorphism, being a bijection of compact spaces. $\Phi$ and $q$ are homomorphisms of monoids; therefore, $p$ is a topological isomorphism.
\end{proof}

\medskip

Let a topological group $G$ be $\tau_p$-representable in $X$, $\mathbb{GC}(X)$ is a poset of $G$-compactifications of $X$. By Lemma~\ref{l1} $G$ is $\tau_p$-representable in $b X$ for any $b X\in \mathbb{GC}(X)$. Thus, the map  
$$\mathfrak{E}_{\theta}:\mathbb{GC}(X)\to \mathbb{E} (G),\ \mathfrak{E}_{\theta}(b X)=e_{b X} G,$$ is defined. 

\begin{rem}\label{compmapsth}
{\rm A topological group $G$ is $\tau_p$-representable in itself (see, \S~\ref{admtop}). Therefore,  the map  
$\mathfrak{E}:\mathbb{GC}(G)\to \mathbb{E} (G)$, $\mathfrak{E} (b G)=e_{b G} G$, is defined, where $\mathbb{CG}(G)$ is the poset of $G$-compactifications of $G$. If $b G\in\mathbb{E}(G)$, then $\mathfrak{E}(b G)=b G$. For other properties of the map $\mathfrak{E}$ see~\cite[Theorem 4.3]{KozlovLeiderman} or~\cite[Theorem 2.22]{KozlovSorin}.

In particular, if $G$ is $\tau_p$-representable in $X$, then $\mathfrak{E}\circ \mathfrak{E}_{\theta}=\mathfrak{E}_{\theta}$, since $\mathfrak{E}_{\theta}( bX)\in\mathbb{E} (G)$, $\forall\ b X\in\mathbb {GC}(X)$.}
\end{rem}

\begin{cor}\label{order} 
Let a topological group $G$ be $\tau_p$-representable in $X$. Then the map $\mathfrak{E}_{\theta}:\mathbb{GC}(X)\to \mathbb{E} (G)$ is a morphism of posets.
\end{cor}

\begin{proof}
If $b X\leq b' X$, then the map of compactifications $\varphi: b'G \to b G$ induces, by Proposition~\ref{mapelliscomp}, the map of Ellis compactifications $\Phi: e_{b' X} G\to e_{b X} G$. Hence, $e_{b' X} G\geq e_{b X} G$. 
\end{proof}

\begin{que} 
Let a topological group $G$ be $\tau_p$-representable in $X$. When the map $\mathfrak{E}_{\theta}$ is injective {\rm(}surjective, bijective{\rm)}?
\end{que}


\subsection{Comparison of uniformities on topological groups}

Let a topological group $G$ be $\tau_p$-representable in the compact space $X$. For the unique uniformity $\mathcal U_X$ on $X$, $\mathcal U_X^p$ is the Cartesian product of uniformities on $X^X$. Then he subspace uniformity  $\mathcal V_G (X)$ on $\imath_X (G)$ as a subset of $(X^X, \mathcal U_X^p)$ is initial with respect to the restrictions to $\imath_X (G)$ of the projections $\pr_{\sigma}: X^X\to (X^{\sigma}, \mathcal U_X^{\sigma})$  (maps to uniform spaces, where products $X^{\sigma}$ are with the Cartesian product of uniformities $\mathcal U_X$), $\sigma\in\Sigma_X$.

\begin{rem}
{\rm Identifying $G$ and $\imath_X (G)$, the uniformity  $\mathcal V_G (X)$ on $G$ is an equiuniformity. 

Since $(\pr_x\circ\imath_X)(g)=\theta (g, x)=\theta_x(g)$, $g\in G$, the equiuniformity $\mathcal V_G (X)$ is the least uniformity $\mathcal U$ on $G$ such that the orbit maps $\theta_x: (G, \mathcal U)\to (X, \mathcal U_X)$, $x\in X$, are uniformly continuous.

The equiuniformity $\mathcal V_G (X)$ on $\imath_X (G)$ is the subspace uniformity on $G$ as a subset of $e_X G$.}
\end{rem}

Let $q_{\sigma}: G\to G/\st_{\sigma}$ be the quotient map, $\sigma\in\Sigma_X$. The following diagram is commutative for any $\sigma\in\Sigma_X$
$$\begin{array}{ccc}
\quad G & \stackrel{\imath_X}\hookrightarrow & X^X  \\
 q_{\sigma}  \downarrow &   & \quad \downarrow \pr_{\sigma}\\
\quad G/\st_{\sigma} & \stackrel{\psi_{\sigma} }\longrightarrow & X^{\sigma}, 
\end{array}\eqno{\rm (d)}$$
where $\psi_{\sigma}  (g \st_{\sigma})(x)=\theta (g, x)$, $x\in\sigma$ ($\psi_{\sigma}$ is correctly defined since if $g'=gh$, $h\in\st_{\sigma}$, then $\theta (g', x)=\theta (gh, x)=\theta (g, \theta (h, x))=\theta (g, x)$). $\psi_{\sigma}$ is a bijection, since $g\st_{\sigma}=h\st_{\sigma}$ iff $g(x)=h(x)$, $x\in\sigma$. Considering all spaces in the diagram (d) as $G$-spaces (with the action $\alpha$ of $G$ on itself, by multiplication on the left in $G/\st_{\sigma}$ and the diagonal actions induced by $\theta$ on products) the maps in the diagram (d) are $G$-maps. 

\begin{pro}\label{lemincl}
Let a topological group $G$ be $\tau_p$-representable in compact $X$, $\mathcal U_X$ is uniformity on $X$. Then for any $\sigma\in\Sigma_X$ the map $\psi_{\sigma}: (G/\st_{\sigma}, \mathcal U_{G/\st_{\sigma}})\to ((\pr_{\sigma}\circ\imath_X) (G), \mathcal U_X^{\sigma}|_{\pr_{\sigma}(\imath_X (G))})$ is uniformly continuous. Hence, $\mathcal V_G (X)\subset R_{\Sigma_X}$.
\end{pro}

\begin{proof}  $\mathcal U_X^p$ is an equiuniformity  on $X^X$ and projection $\pr_{\sigma}$ is uniformly continuous, $\sigma\in\Sigma_X$. Therefore, the map $\pr_{\sigma}\circ\imath_X: (G, R)\to (X^{\sigma}, \mathcal U_X^{\sigma})$ is uniformly continuous. The maximal equiuniformity $\mathcal U_{G/\st_{\sigma}}$ on $G/\st_{\sigma}$ is the quotient of the right uniformity $R$ on $G$. Hence, the map $\psi_{\sigma}$, being a bijection, is uniformly continuous.

Since $R_{\Sigma_X}$ is initial with respect to the quotient maps $q_{\sigma}: G\to (G/\st_{\sigma}, \mathcal U_{G/\st_{\sigma}})$,  see \S~\ref{unifgr}, and $\mathcal V_G(X)$ is initial with respect to the restrictions to $\imath_X (G)$ of the projections $\pr_{\sigma}: X^X\to (X^{\sigma}, \mathcal U_X^{\sigma})$ one has $\mathcal V_G (X)\subset R_{\Sigma_X}$. 
\end{proof}

\begin{cor}\label{coinunif1} Let a topological group $G$ be $\tau_p$-representable in compact $X$, $\mathcal U_X$ is uniformity on $X$. If the uniformities on $G$ which are initial with respect to the quotient maps $q_{\sigma}: G\to (G/\st_{\sigma}, \mathcal U_{G/\st_{\sigma}})$, $\sigma\in\Sigma_X$, and the maps $\pr_{\sigma}|_{\imath (G)}\circ\imath_X: G\to (X^{\sigma}, \mathcal U_X^{\sigma})$,  $\sigma\in\Sigma_X$, respectively, coincide 
$${\rm(}\mbox{in particular},\ \psi_{\sigma}\ \mbox{is a uniform isomorphism},\  \sigma\in\Sigma_X,{\rm)}\eqno{\rm (op)}$$
then $\mathcal V_G (X)=R_{\Sigma_X}$.
\end{cor}

\begin{rem}{\rm 
Since the quotient maps  $q_{\sigma}: G\to G/\st_{\sigma}$ are open, if the condition (op) is valid, then the restriction of projection $\pr_{\sigma}|_{\imath_X (G)}$ to $\imath_X (G)$ is an open map, $\sigma\in\Sigma_X$.}
\end{rem}

\medskip

If a topological group $(G, \tau)$ is $\tau_p$-representable in the compact space $X$, $\mathcal U_X$ is the uniformity on $X$, then $\tau=\tau_{co}$ (for the action $\theta: G\times X\to X$). From~\cite[Ch.\,2, Ex.\,2]{RD} it follows that the base of the right uniformity $R$ on $G$ is formed by entourages
$$\{(g, h)\in G\times G\ |\ (\theta (g, x), \theta (h, x))\in{\rm U},\ x\in X\},\ {\rm U}\in\mathcal U_X.$$
Therefore, the base of $N_G(e)$ is formed by internals of the sets 
$$O_{{\rm U}}=\{g\in G\ |\ (x, \theta (g, x))\in  {\rm U},\ x\in X\},\ {\rm U}\in\mathcal U_X.$$
$O^{-1}=O$, since $(x, \theta (g^{-1}, x))\in {\rm U}$, $x\in X$, iff  $(y, \theta (g, y))\in {\rm U}$, $y{\rm(}=\theta (g^{-1}, x){\rm)}\in X$.

\begin{pro}\label{a2-2}
Let a topological group $G$ be $\tau_p$-representable in compact $X$, $\mathcal U_X$ is the uniformity on $X$.

If for any  $\sigma\in\Sigma_{X}$ and for any entourage ${\rm U}\in {\mathcal U}_X$ there is an entourage ${\rm V}={\rm V}(\sigma; {\rm U})\in {\mathcal U}_X$ such that the condition is met:
$$\begin{array}{c}
\mbox{\rm if for}\ g, h\in G\ (\theta(g, x), \theta(h, x))\in {\rm V}\ \mbox{\rm holds for}\ x\in\sigma, \\
\mbox{\rm then there exists}\ \ g'\in g\st_{\sigma}\ \ \mbox{\rm such that}\ \ h\in O_{{\rm U}}g',\\ 
\end{array} \eqno{(\star\star)}$$
where $O_{{\rm U}}=\{f\in G\ |\ (x, \theta(f, x))\in{\rm U},\ x\in X\}$.
Then $\mathcal V_G (X)=R_{\Sigma_{X}}$.
\end{pro}
 
\begin{proof}
By Corollary~\ref{coinunif1} it is sufficient to show that for any $\sigma\in\Sigma_{X}$  the map $\psi_{\sigma}$ is a uniform isomorphism.
Since $\mathcal V_G (X)\subset R_{\Sigma_X}$ by Proposition~\ref{lemincl}, only the uniform continuity of $\psi_{\sigma}^{-1}$ must be checked. One must show that in the cover $$\{O_{{\rm U}}g\st_{\sigma}\ |\ g\in G\},\ {\rm U}\in {\mathcal U}_X,$$
the cover 
$$\{U_{\sigma; g; {\rm V}}=\{h\in G\ |\ (\theta(g, x), \theta(h, x))\in {\rm V},\ x\in\sigma\}\ |\ g\in G\},$$
is refined for some ${\rm V}\in{\mathcal U}_X$. 
 
By the condition $(\star\star)$ for $\sigma\in\Sigma_{X}$ and ${\rm U}\in {\mathcal U}_X$ there exists ${\rm V}\in {\mathcal U}_X$ such that if for $g, h\in G$ $(\theta(g, x), \theta(h, x))\in {\rm V}$, $x\in\sigma$, then there exists $g'\in g\st_{\sigma}$ such that $h\in O_{{\rm U}}g'$. Then for any fixed $g\in G$ and any $h\in U_{\sigma; g; {\rm V}}$ 
$$h\in O_{{\rm U}} g\st_{\sigma}.$$
Therefore, $U_{\sigma; g; {\rm V}}\subset   O_{{\rm U}}g\st_{\sigma}$, and the cover   
$\{U_{\sigma; g; {\rm V}}\ |\ g\in G\}$ is refined in the cover $\{O_{{\rm U}}g\st_{\sigma}\ |\ g\in G\}$. Hence, $\psi_{\sigma}$ is a uniform isomorphism.
\end{proof}

\begin{rem}\label{Rlkcomp}
{\rm 
If $\mathcal V_G (X)=R_{\Sigma_X}$ ($X$ is compact), in particular, the condition $(\star\star)$ is valid, then $G$ is Roelcke precompact by Remark~\ref{a2-1} (see~\cite[Corollary 2.2]{Sorin1} or~\cite[Theorem 1]{Sorin}) and $b_r G\leq e_X G$ (since $L\wedge R\subset R_{\Sigma_X}$).

\medskip

The condition $h\in O_{{\rm U}}g'$ in Proposition~\ref{a2-2} is equivalent to the condition $(\theta(h, x), \theta(g', x))\in {\rm U}$, $x\in X$. Indeed, $(\theta(h, x), \theta(g', x))\in {\rm U},\ x\in X\Longleftrightarrow (\theta (h{g'}^{-1}, x), x)\in {\rm U},\ x\in X \Longleftrightarrow h{g'}^{-1}\in O_{{\rm U}} \Longleftrightarrow  h\in O_{{\rm U}}g'$.

\medskip

Proposition~\ref{a2-2} is a part of~\cite[Theorem 2.1]{Sorin1} or \cite[Theorem 1]{Sorin} (see, also~\cite[Proposition 2.16]{KozlovSorin}.}
\end{rem}

\begin{rem}{\rm If a topological group $G$ is $\tau_p$-representable in compact $X$ and the action  $G\curvearrowright X$ is (uniformly)  equicontinuous, then the Ellis compactification $e_X G$ is a compact topological group~\cite[Theorem 4.2]{Vries} (see, also~\cite[Ch.10,\ \S\ 3]{Burb} and \cite[Theorem 3.33]{Kozlov}).}
\end{rem}

A (closed) subgroup $H$ of a topological group $G$ is called {\it neutral} if $\forall\ O\in N_{G}(e)$ $\exists\ V\in N_{G}(e)$ such that $VH\subset HO$~\cite[Definition 5.29]{RD}. 

Let $H$ be a neutral subgroup of a topological group $(G, \tau)$ without nontrivial invariant subgroups. Then $G$ is a uniformly equicontinuous group of homeomorphisms of $((G, \tau)/H, \mathcal L)$, where $\mathcal L$ is the quotient uniformity on coset space $(G, \tau)/H$ of the left uniformity $L$ on $G$ (the final uniformity on $(G, \tau)/H$ with respect to the quotient map $(G, L)\to G/H$~\cite[Proposition 5.40]{RD}). $\tau_p$ is an admissible group topology on $G$ for action $G\curvearrowright (G, \tau)/H$, $\tau_p\leq\tau$. By~\cite[Proposition 3.2]{Kozlov} $H$ is a neutral subgroup of $(G, \tau_p)$ and 
$((G, \tau_p), (G, \tau)/H, \theta)$ is a $G$-Tychonoff space (with an open action by multiplication on the left). Hence, $(G, \tau_p)$ is $\tau_p$-representable in the coset space $(G, \tau)/H=(G, \tau_p)/H$.

\begin{pro}\label{autultr} Let $H$ be a neutral subgroup without nontrivial invariant subgroups of a topological group $G$, $G$ is $\tau_p$-representable in the coset space $X=G/H$ and the maximal equiuniformity  $\mathcal U^{max}_{X}$ on $X$ is totally bounded. Then 
\begin{itemize}
\item[(1)] ${\mathfrak E} (b_r G)\geq {\mathfrak E}_{\theta} (\beta_G X)$, 
\item[(2)] if ${\mathfrak E}_{\theta} (\beta_G X)\geq b_r G$, then  ${\mathfrak E}_{\theta} (\beta_G X)={\mathfrak E} (b_r G)$.
\end{itemize}
\end{pro}

\begin{proof} (1) Let  $q: G\to X$ be the quotient map. $\mathcal U^{max}_{X}$ is the quotient uniformity of the right uniformity $R$ on $G$. By~\cite[Proposition 3.25]{Kozlov} $\mathcal U^{max}_{X}=\mathcal L\wedge\mathcal R$, where $\mathcal L\wedge\mathcal R$ is the quotient uniformity of the Roelcke uniformity $L\wedge R$ on $G$. 

The surjective $G$-map $q: (G,  L\wedge R)\to (X, \mathcal L\wedge\mathcal R)$ is uniformly continuous. Since $\mathcal L\wedge\mathcal R$ is totally bounded, for the precompact replica $(L\wedge R)_{fin}$ of $L\wedge R$ the map  $q: (G, (L\wedge R)_{fin})\to (X, \mathcal L\wedge\mathcal R)$ is uniformly continuous and can be extended to an onto $G$-map  $\tilde q: b_r G\to\beta_G X$. By Proposition~\ref{mapelliscomp} ${\mathfrak E} (b_r G)\geq {\mathfrak E}_{\theta} (\beta_G X)$. 

\medskip

(2) If ${\mathfrak E}_{\theta} (\beta_G X)\geq b_r G$, then $\mathfrak{E}({\mathfrak E}_{\theta} (\beta_G X))={\mathfrak E}_{\theta} (\beta_G X)$ by Remark~\ref{compmapsth} and ${\mathfrak E}_{\theta} (\beta_G X)\geq\mathfrak{E} (b_r G)$ by Corollary~\ref{order}. From (1) the equality ${\mathfrak E}_{\theta} (\beta_G X)=\mathfrak{E} (b_r G)$ follows.
\end{proof}
 
\begin{rem}\label{nocomp}
{\rm If ${\mathfrak E}_{\theta} (\beta_G X)=e_{\beta_G X} G\geq b_r G$, then there are no Ellis compactifications $b G$ of $G$ such that $b_r G< b G<e_{\beta_G X} G$ (see, for instance,~\cite[Theorem 4.3]{KozlovLeiderman} or~\cite[Theorem 2.22]{KozlovSorin}).}
\end{rem}


\section{Permutation group of a discrete space}\label{permutat}

Let $X$ be a discrete (infinite) space, $G$ is a subgroup of the permutation group ${\rm S}(X)$. A discrete space $X$ can be considered as a metric space with distance $1$ between distinct points. Thus ${\rm S}(X)$ is an isometry group and $G$ is its subgroup. The topology of pointwise convergence $\tau_p$ is the least admissible group topology on $G$ for the action $\theta: G\times X\to X$, $\theta (g, x)= g(x)$, $g\in G$, $x\in X$. The base of the nbds of the unit $e$ of $(G, \tau_p)$ is formed by clopen subgroups 
$$\st_{\sigma}=\{g\in G\ |\ \theta (g, x)=x,\ x\in\sigma\},\  \sigma\in\Sigma_X.$$

\begin{thm}\label{Roelcke precomp4-1} Let $X$ be a discrete space and $G=({\rm S}(X), \tau_p)$. Then 
\begin{itemize}
\item[{\rm (a)}] the Alexandroff one-point compactification $\alpha X=X\cup\{\infty\}$ of $X$ is the unique $G$-compactification of $X$,
\item[{\rm (b)}] $G$ is Roelcke precompact and $e_{\alpha X} G=b_r G$, 
\item[{\rm (c)}]  $b_r G$ is the set of self-maps $f$ of $\alpha X$ in the topology of pointwise convergence such that
$f$ is a bijection on $Y\subset X$ and $f(\alpha X\setminus Y)=\infty$, $Y$ is an arbitrary subset of $X$, 
\item[{\rm (d)}] $b_r G$ is a sim-compactification of $G$, isomorphic to the symmetric inverse semigroup $I_X$ on $X$ and a WAP-compactification of $G$.
\end{itemize}
\end{thm}

\begin{proof} (a) $X=G/\st_x$, $x\in X$. The base of the maximal equiuniformity $\mathcal U^{max}_X$ on $X$ is formed by the uniform covers 
$$u_{\sigma}=\{\st_{\sigma}x=\theta (\st_{\sigma}, x)\ |\ x\in X\},\ \sigma\in\Sigma_X,\ {\rm(}u_{\sigma'}\succ u_{\sigma}\ \mbox{if}\ \sigma\leq\sigma'{\rm )},$$
see, for example, \cite{ChK}. Since $\st_{\sigma}x=x$ if $x\in\sigma$ and  $\st_{\sigma}x=X\setminus\sigma$ if $x\notin\sigma$, they are of the  form 
$$u_{\sigma}=\{\{\{x\}\ |\ {x\in\sigma}\},\ X\setminus\sigma\},\ \sigma\in\Sigma_X,$$
and coincides with the subspace uniformity on $X$ as the subset of $\alpha X$. Hence, $\alpha X$ is the maximal (and the unique) $G$-compactification of $X$. 

\medskip

(b) By Remark~\ref{a2-1} $L\wedge R=R_{\Sigma_X}$. Since  $\st_{\infty}=G$ for the extended action $\tilde\theta: G\times\alpha X\to\alpha X$ (continuity of each $g\in G$ yields that $\tilde\theta (g, \infty)=\infty$),  $R_{\Sigma_X}=R_{\Sigma_{\alpha X}}$.

By Proposition~\ref{lemincl} $\mathcal V_G (\alpha X)\subset R_{\Sigma_X}$. To end the proof of (b) it is enough to check the inclusion $\mathcal V_G (\alpha X)\supset R_{\Sigma_X}$. In the case of the equiuniformity  $\mathcal U^{max}_X$ the condition $(\star\star)$ in Proposition~\ref{a2-2} can be reformulated in the following form 
$$\begin{array}{c}
\mbox{\rm if for}\ \sigma\subset\sigma'\in\Sigma_X\ \mbox{and}\ g, h\in G,\ \mbox{the inclusion}\ \theta (h, x)\in\st_{\sigma'}\theta (g, x),\  x\in\sigma,\ \mbox{\rm holds},\\
\mbox{\rm then there exists}\ \ g'\in g\st_{\sigma}\ \ \mbox{\rm such that}\ \ h\in\st_{\sigma'}g'.\\ 
\end{array}$$
For any $g, h\in G$ such that $\theta (h, x)\in\st_{\sigma'}\theta (g, x)$, $x\in\sigma$, we consider, without loss of generality, the following. 
 
Put $s=\{x\in\sigma\ |\ \theta (g, x)\in\sigma'\}\subset\sigma$. Then $\theta (h, x)=\theta (g, x)\in\sigma'$, $x\in s$. If $s\ne\sigma$, then $\theta (g, x)\not\in\sigma'$, $\theta (h, x)\not\in\sigma'$ for $x\in\sigma\setminus s$.

If $s=\sigma$, then $hg^{-1}\in\st_{\sigma'}$ and $h\in\st_{\sigma'}g$, $g\in g\st_{\sigma}$.

\medskip 
 
If $s\ne\sigma$, then there exists $f\in G$ such that the points $\sigma'\cup\{\theta (g, y)\ |\ y\in\sigma\setminus s\}$, are mapped to the points $\sigma'\cup\{\theta (h, y)\ |\ y\in\sigma\setminus s\}$  respectively (for fixed orders on $\sigma'$ and $\sigma\setminus s$).  $f\in\st_{\sigma'}$ and $\theta (h, x)=\theta (fg, x),\ x\in\sigma$. 
Hence, $h^{-1}f g\in\st_{\sigma}$. Consequently $f g\in h\st_{\sigma}$ (equivalently $h\in (fg)\st_{\sigma}$) and there is $g'\in g\st_{\sigma}$ such that $h\in \st_{\sigma'}g'$. The  inclusion  $\mathcal V_G (\alpha X)\supset R_{\Sigma_X}$ and, hence, the equalities $L\wedge R=\mathcal V_G (\alpha X)$, $e_{\alpha X} G=b_r G$ are proved. By Remark~\ref{Rlkcomp} $G$ is Roelcke precompact. 

\medskip

(c) At first, let us note that for any $f\in e_{\alpha X} G=\cl~\imath_{\alpha X} (G)$ there are no points $x\ne y\in X$ such that $f(x)=f(y)\in X$. Indeed, suppose that there are $x\ne y\in X$ such that $f(x)=f(y)\in X$. One can take the nbd $W=\{h\in e_{\alpha X} G\ |\ h(x)=h(y)=f(x)\}$ of $f$. Then $W\cap\imath_{\alpha X} (G)=\emptyset$ and $f\not\in e_{\alpha X} G$. 

Secondly, $f(\infty)=\infty$ for any $f\in e_{\alpha X} G$. Indeed, if $f(\infty)=x\in X$, then take the nbd $W=\{h\in e_{\alpha X} G\ |\ h(\infty)=x\}$ of $f$. However, $W\cap\imath_{\alpha X} (G)=\emptyset$ since $g(\infty)=\infty$ for any $g\in\imath_{\alpha X} (G)$.

It remains to check that $f\in e_{\alpha X} G$ if $f(\infty)=\infty$, $f$ is a bijection on $Y\subset X$ and $f(X\setminus Y)=\infty$, where $Y$ is an arbitrary subset of $X$. Any nbd $W$ of $f$ is of the form $\{g\in (\alpha X)^{\alpha X}\ |\ g(x)\in O_{f(x)},\ x\in\sigma\}\cap e_{\alpha X} G$, $\sigma\in\Sigma_X$, where $O_{f(x)}=f(x)$ if $f(x)\ne\infty$, and $O_{f(x)}=\alpha X\setminus\sigma_x$, $\sigma_x\in\Sigma_X$, if $f(x)=\infty$. By finiteness of $\sigma$ and $\sigma_x$, $x\in\sigma$, there is $h\in G$ such that $h(x)=f(x)$ if $x\in\sigma$, $f(x)\in X$, and $h(x)\not\in\sigma_x$, if $x\in\sigma$, $f(x)=\infty$. Evidently, $h\in W$ and, hence, $f\in e_{\alpha X} G$.

\medskip

(d) From (b) it follows that $b_r G=e_{\alpha X} G$ is a right topological semigroup. If one shows that all elements $f\in e_{\alpha X} G$ are continuous (as selfmaps of $\alpha X$), then multiplication on the left is continuous. To check the continuity of $f$ it is sufficient to check the continuity of $f$ at the point $\infty$.

Suppose that $f$ is not continuous at the point $\infty$. Then there exists a nbd $W=\alpha X\setminus\sigma$ of $\infty$, $\sigma\in\Sigma_X$, such that the set $\{x\in X\ |\ f(x)\in\sigma\}$ is infinite.  Since $\sigma$ is finite, there exist $x\ne y\in X$ such that $f(x)=f(y)$. Therefore,  $f\not\in e_{\alpha X} G$. The obtained contradiction shows that all  $f\in e_{\alpha X} G$ are continuous and $b_r G$ is a semitopological semigroup.

For $f\in  b_r G$ put $D(f)=\{x\in X\ |\ f(x)\ne\infty\}$. Define a map $\xi: b_r G\to I_X$, $\xi (f)=f|_{D(f)}$. 
$$\xi (f\circ g)=(f\circ g)|_{D(g)\cap g^{-1}(D(f))}=f|_{g(D(g)\cap g^{-1}(D(f)))\cap D(f)}\circ g|_{D(g)\cap g^{-1}(D(f))}=$$
$$=f|_{D(f)\cap g(D(g))}\circ g|_{D(g)\cap g^{-1}(D(f))}=f|_{D(f)}\circ g|_{D(g)}=\xi(f)\circ\xi (g).$$
Hence, $\xi$ is a homomorphism. If $\xi (f)=\xi (g)$, then $f|_{D(f)}=g|_{D(g)}$ and, therefore, $f=g$. Evidently $\xi$ is a surjection. Hence, $\xi$ is an isomorphism.

Take $f\in b_r G$, then 
$$f^*=\left\{\begin{array}{ccl}
f^*(x)=y & \mbox{if} & f(y)=x,\ x, y\in X, \\
f^*(x)=\infty & \mbox{if} & x\not\in f(D(f)), \\
f^*(\infty)=\infty &  &  \\
\end{array}\right.$$
is an inverse of $f$ due to isomorphism $\xi$. The nbd of $f^*$ from the subbase of $\tau_p$ is either $[x, f^*(x)]$ for $x\in D(f^*)$, or $[x, \alpha X\setminus \sigma]$, $\sigma\in\Sigma_X$, for $x\not\in D(f^*)$.

For a nbd $[x, y=f^*(x)]$ of $f^*$ take an open subset $[y, x]$. Then $f\in [y, x]$ and for any $g\in [y, x]$ $g^*\in [x, y]$. 

For a nbd $[x, \alpha X\setminus \sigma]$ of $f^*$, $f^*(x)=\infty$, take a nbd $W=\bigcap\{[y, \alpha X\setminus\{x\}]\ |\ y\in\sigma\}$ of $f$ ($x\not\in f(D(f))$, therefore, $f\in W$). If $g\in W$, then $g(y)\ne x$, $y\in\sigma$. Hence, $g^*(x)\not\in\sigma$ and $g^*\in W$.  Therefore, inversion is continuous. 

Any sm-compactification of $G$ is less than or equal to $b_r G$~\cite{KozlovLeiderman}. Since $b_r G$ is a sim-compactification of $G$, $b_r G$ is the maximal sm-compactification of $G$ and a WAP-compactification of $G$.
\end{proof}

From~\cite{Malcev} we have the description of all proper ideals of  the symmetric inverse semigroup $I_X$ (where $\emptyset$ is $0$): 
$$I_X^{\delta}=\{f\in I_X\ |\ |{\rm D} (f)|\leq\delta\},\  1\leq\delta<|X|,\ I_X^{0}=\emptyset.$$ 
Ideals $I_X^{\delta}$ are invariant under the involution defined by the inverse. 

\begin{cor}\label{sisdiscr}
In $b_r G$ {\rm(}$G=({\rm S} (X), \tau_p)${\rm)}   
$$I_X^n=\{f\in b_r G\ |\ |D(f)|\leq n\},\ n\in\mathbb N,$$
are compact ideals. The Rees quotients $b_r G/\rho_{I_X^n}$, $n\in\mathbb N$, are sim-compactifications of $G$.

If $X$ is countable, then sim-compactifications $b_r G$ and $b_r G/\rho_{I_X^n}$, $n\in\mathbb N$, of $G$ are metrizable {\rm(}homeomorphic to the Cantor set{\rm)}.
\end{cor}

\begin{proof}
$\forall\ n\in\mathbb N$ $I_X^n$ is a compact ideal. Indeed, if $f\in I_X$ and $D(f)>n$, then $W=\bigcap\limits_{i=1}^{n+1}\{[x_i, f(x_i)]\ |\ x_i\in D(f)\}$ is a nbd of $f$ and $f\cap I_X^n=\emptyset$. $b_r G/\rho_{I_X^n}$, $n\in\mathbb N$, are sim-compactifications of $G$ by Lemma~\ref{ideals}. 

If $X$ is a countable set, then $\alpha X$ is metrizable and countable, and, therefore, $(\alpha X)^{\alpha X}$ is metrizable. $b_r G$ is metrizable as a subset of $(\alpha X)^{\alpha X}$ and $b_r G/\rho_{I_X^n}$ is metrizable as perfect image of $b_r G$, $n\in\mathbb N$. Further, $b_r G$ and $b_r G/\rho_{I_X^n}$, $n\in\mathbb N$, are zero-dimensional and dense in themselves compact spaces, homeomorphic to the Cantor set.
\end{proof}

\begin{rem}
{\rm  An action of a group $G$ on a set $X$ is strongly $n$-transitive, $n\geq 1$, if for any families of distinct $n$ points $x_1, \ldots, x_n$ and $y_1, \ldots, y_n$ there exists $g\in G$ such that $g(x_k)=y_k$, $k=1,\ldots, n$. 

An action $G\curvearrowright X$, which is strongly $n$-transitive for all $n\in\mathbb N$, is called ultratransitive.

\medskip

(1) Let $G$ be a subgroup of ${\rm S} (X)$ which acts ultratransitively on $X$. Then $G$ is a dense subgroup 
of $({\rm S} (X), \tau_p)$~\cite[Fact 2 (2)]{Sorin2025} and $b_r G$ is topologically isomorphic to $b_r {\rm S} (X)$.

\medskip

(2)  Any group $G$ which acts ultratransitively on $X$ is {\it oligomorphic} and, hence, Roelcke precompact~\cite{Tsan} (see, also~\cite{Sorin2025}). 

\medskip

(3) The subgroup ${\rm S}_{<\omega}(X)$ of the group ${\rm S}(X)$ whose elements has finite supports acts ultratransitively on $X$. Hence, ${\rm S}_{<\omega} (X)$ is a dense subgroup of $({\rm S}(X), \tau_p)$.  The Roelcke precompactness of $({\rm S}_{<\omega}(X), \tau_p)$ is proved in~\cite{Ban}. 

The group whose elements have finite supports is a subgroup of any Houghton's group (see, for example, \cite{Cox}). Hence, Houghton's groups act ultratransitively on the corresponding countable discrete spaces and are oligomorphic.}
\end{rem}


\section{Automorphism groups of ultrahomogeneous chains and LOTS}\label{chain}

For an infinite {\it chain} ({\it linearly ordered set}) $X$ let $\aut(X)$ be the group of {\it automorphisms} (order-preserving bijections) of $X$. A chain $X$ is called {\it homogeneous} if the action $\aut(X)\curvearrowright X$ is transitive {\rm(}for any $x, y\in X$ there exists $f\in\aut(X)$ such that $f(x)=y${\rm)}. 
A chain $X$ is called {\it ultrahomogeneous} if for any families of different $n$ points $x_1<\ldots< x_n$ and   $y_1<\ldots< y_n$ there exists $g\in\aut(X)$ such that $g(x_k)=y_k$, $k=1,\ldots, n$, $n\in\mathbb N$.

\begin{rem}\label{propchain}
{\rm 
(1) An chain $X$ is ultrahomogeneous iff $X$ is {\rm 2}-{\it homogeneous} (for any pairs of points  $x<y$ and $x'<y'$ there exists $g\in\aut(X)$ such that $g(x)=x'$, $g(y)=y'$)~\cite[Lemma 1.10.1]{Glass} (see also~\cite{Ovch}). 

(2) An ultrahomogeneous chain is {\it densely ordered}~\cite[\S\ I.3]{Engelking} (and, hence, has no {\it jumps}). 

(3) A densely ordered chain without {\it proper gaps} is called {\it continuously ordered}~\cite[\S\ I.3]{Engelking}. If a chain $X$ is ultrahomogeneous and has a proper gap, then there is a gap on every interval $J$ of $X$ ($J\subset X$ is an {\it interval} if $J=(a, b)=\{x\in X\ |\ a<x<b\}$, $a<b\in X$, $J=(\gets, b)=\{x\in X\ |\ x<b\}$,  $J=(a, \to)=\{x\in X\ |\ a<x\}$,  $a, b\in X$, or $J=X$).

(4) Let $X$ be a {\it discrete chain} (a chain in the discrete topology). On the group $\aut(X)$, the {\it permutation topology} $\tau_{\partial}$ (topology of pointwise convergence for the action on a discrete space) is the least admissible group topology for the uniformly equicontinuous action (action by isometries) $\aut(X)\curvearrowright X$, $\st_x$ is a neutral subgroup, $x\in X$. $(\aut(X), \tau_{\partial})$ is a subgroup of $({\rm S}(X), \tau_p)$. 

If $X$ is an ultrahomogeneous discrete chain, then $X=(\aut(X), \tau_{\partial})/\st_x$, $x\in X$ (see, for instance, \cite{Kozlov}). 

(5) Let $X=(X, \tau)$ be LOTS ({\it linearly ordered space}), the topology $\tau$ is induced by linear order. The topology of pointwise convergence $\tau_p$ is the least admissible group topology for the action $\theta: \aut(X)\times X\to X$~\cite{Ovch} or~\cite{Sorin}.

If $X$ is an ultrahomogeneous LOTS, then $X=(\aut (X), \tau_p)/\st_x$, $x\in X$. Indeed, fix $x\in X$ and put $q: \aut(X)\to X$, $q(g)=\theta (g, x)$. Since $q(h)=q(g)\Longleftrightarrow h\in g\st_x$ and $q(G)=X$, the continuous bijection $\aut (X)/\st_x\to X$ is defined. Since $q$ is open (if $O=(a, b)$ and $y\in X$, then $q([y, O])=\theta ([y, O], x)=(\gets, b)$ if $x<y$, $O$ if $x=y$ and $(a, \to)$ if $x>y$), $X$ and $(\aut (X), \tau_p)/\st_x$ are homeomorphic.

(6) If $X$ is a LOTS, then, according to the description in~\cite{Fed}, the least linearly ordered compactification $c_m X$ \cite{Kauf} is generated by replacing each gap in $X$ by a point with a natural continuation of the order. If $X$ is densely ordered, then $c_m X$ is connected. If $X$ is continuously ordered, then $c_m X=\{\inf\}\cup X\cup\{\sup\}$ is the unique linearly ordered compactification of $X$. 

(7) The ultrahomogeneity of chain $X$ is equivalent to the oligomorphicity of the action of $\aut (X)\curvearrowright X$, and is also equivalent to the Roelcke precompactness of the group $\aut(X)$ (both in the permutation topology and in the topology of pointwise convergence)~\cite{Sorin2025}.}
\end{rem}

\begin{df} An action $G\curvearrowright X$ of a subgroup $G$ of $\aut (X)$ on a chain $X$ is strongly $n$-transitive, $n\geq 1$, if for any families of $n$ points $x_1< \ldots<x_n$ and $y_1<\ldots<y_n$ there exists $g\in G$ such that $g(x_k)=y_k$, $k=1,\ldots, n$. 

An action $G\curvearrowright X$, which is strongly $n$-transitive for all $n\in\mathbb N$, is called ultratransitive.
\end{df}

\begin{rem}\label{denseultr}
{\rm If an action $\theta: G\times X\to X$ on a discrete chain $X$ is ultratransitive, then $G$ is a dense subgroup $(\aut (X), \tau_{\partial})$ and $X$ is a coset space of  $(G, \tau_{\partial})$.

Indeed, for any $g\in\aut (X)$ and a nbd $W=\bigcap\{[x, g(x)]\ |\ x\in\sigma\}$, $\sigma\in\Sigma_X$, of $g$, ultratransitivity of the action $G\curvearrowright X$ yields that there exists $f\in G\cap W$. The rest follows from item (4) of Remark~\ref{propchain}.

\medskip

If an action $G\curvearrowright X$ on a LOTS $X$ is ultratransitive, then $G$ is a dense subgroup $(\aut (X), \tau_p)$ and $X$ is a coset space of  $(G, \tau_p)$.

Indeed, for any $g\in\aut (X)$ and a nbd $W=\bigcap\{[x, O_{g(x)}]\ |\ x\in\sigma\}$, $\sigma\in\Sigma_X$, $O_{g(x)}$ is a nbd of $g(x)$, $x\in\sigma$, of $g$, ultratransitivity of the action $G\curvearrowright X$ yields that there exists $f\in G\cap W$. The rest can be proved similarly as in item (5) of Remark~\ref{propchain}.}
\end{rem}

\subsection{$G$-compactifications of ultrahomogeneous discrete chains and LOTS}\label{compact}

A discrete chain $X$ is a GO-space ({\it generalized ordered space}). If $X$ is ultrahomogeneous, then there is the smallest LOTS 
$$X\otimes_{\ell}\{-1, 0, 1\}$$
(topology of linear order is induced by the lexicographic order on $X\times\{-1, 0, 1\}$),  
in which $X$ is a dense subspace, and $X\otimes_{\ell}\{-1, 0, 1\}$ is naturally embedded in any other LOTS in which $X$ is dense~\cite{Miwa}. Hence, any linearly ordered compactification of $X$ is a linearly ordered compactification of $X\otimes_{\ell}\{-1, 0, 1\}$. 
The least linearly ordered compactification $b_m X$ of $X\otimes_{\ell}\{-1, 0, 1\}$ (and, hence, of $X$) is generated by replacing each gap (also improper) by a point with a natural continuation of the order (gaps in $X\otimes_{\ell}\{-1, 0, 1\}$ and  $X$ are naturally identified). Therefore, $b_m X=X^+\cup X^0\cup X^-\cup\Gamma$, where $X^+=X\times\{1\}$, $X^0=X\times\{0\}$ ($X$ and $X^0$ are identified), $X^-=X\times\{-1\}$, $\Gamma$ is the set of gaps. Each gap is defined by a {\it cut} $(A, B)$ ($A, B\subset X$, $A\cap B=\emptyset$, $A\cup B=X$, if $x\in A$, $y\in B$, then $x<y$)~\cite[\S\ I.3]{Engelking}.  $b_m X$ is a zero-dimensional compact LOTS. If $X$ is continuously ordered, then $b_m X=\{\inf\}\cup X^+\cup X^0\cup X^-\cup\{\sup\}$ is the unique linearly ordered compactification of $X$.

\begin{pro}\label{betaG}
Let $X$ be an ultrahomogeneous chain, and $\aut (X)$ is its automorphism group. 

{\rm (a)} If $X$ is a {\rm LOTS (}$\theta_p: (\aut (X), \tau_p)\times X\to X${\rm)}, then $c _m X=\beta_G X$.

{\rm (b)} If $X$ is a discrete chain {\rm(}$\theta_{\partial}: (\aut (X), \tau_{\partial})\times X\to X${\rm)}, then $b _m X=\beta_G X$.
\end{pro}

\begin{proof}
(a) Ultrahomogeneity of $X$ yields that the base of the maximal equiuniformity $\mathcal U^{max}_X$ on LOTS $X$ is formed by finite covers 
$$\{Ox: x\in X\}=\{(\gets, y_2),\ (y_1, y_2),\ (y_1, y_4),\ \ldots,\ (y_{2n-3}, y_{2n}),\ (y_{2n-1}, y_{2n}),\ (y_{2n-1}, \to)\},$$
for  $y_1<x_1<y_2<\ldots<y_{2n-1}<x_n<y_{2n}$, $n\in\mathbb N$, $O=\bigcap\limits_{k=1}^n  [x_k, (y_{2k-1}, y_{2k})]$  (see, for instance, item (5) of Remark~\ref{propchain} and~\cite{ChK}). Hence, $\mathcal U^{max}_X$ is a totally bounded equiuniformity.

In any open cover of a connected (one-dimensional) linearly ordered compact space $c_m X$, one can refine a finite cover $\Omega$ from the intervals $I_1=(\gets, b_1), \ldots, (a_m, b_m),\ldots, I_k=(a_k, \to)$, $m=2, \ldots, k-1$, $a_2<b_1< a_3<b_2<\ldots<a_k<b_{k-1}\in X$ (it is enough to refine in a cover a cover from intervals with end-points in $X$ and choose a minimal subcover from the refined cover in the sense that any of its subsystems is not a cover). Then the cover 
$$\Omega\wedge X=\{(\gets, b_1),\ (a_2, b_1),\ (a_2, b_2),\ \ldots,\ (a_{k-1}, b_{k-1}),\ (a_k, b_{k-1}),\ (a_k, \to)\}$$
of $X$ belongs to $\mathcal U^{max}_X$ (if an arbitrary point $x_m$ is selected in the interval $(a_m, b_{m-1})$, $m=2,\ldots, k$, then $\Omega\wedge X=\{Ox: x\in X\}$ where $O=\bigcap\limits_{m=2}^k [x_m, (a_m, b_{m-1})]$). Since any cover of $\mathcal U^{max}_X$ extends to the cover of $c_m X$, $\beta_G X=c_m X$.

\medskip

(b) Ultrahomogeneity of $X$ yields that the base of the maximal equiuniformity $\mathcal U^{max}_X$ on discrete chain $X$ is formed by finite covers
$$\{\st_{\sigma}x\ |\ x\in X\}=\{(\gets, x_1),\ \{x_1\},\ (x_1, x_2),\ \{x_2\},\ \ldots,\ (x_{n-1}, x_n),\ \{x_n\},\ (x_n, \to)\},$$  $x_1<\ldots<x_n$,  $\sigma=\{x_1,\ldots, x_n\}\in\Sigma_X$. Hence, $\mathcal U^{max}_X$ is a totally bounded equiuniformity.

Since the compactification $b_m X$ is zero-dimensional, in its any open cover it is possible to refine a finite cover of pairwise distinct clopen intervals no none is a subset of the other. By introducing the order on the intervals ($(a_1, b_1)< (a_2, b_2)$ if $a_1< a_2$), we obtain a disjoint cover of the clopen intervals, subtracting sequentially from the $i$ interval the union of the previous ones. Let us correct the latter cover by removing endpoints of the form $(x, 0)$ from the intervals (if any), and adding them as single-point clopen intervals to the corrected interval system to obtain an open cover $\Omega$ of the space $b_m X$. Then $\Omega\wedge X$ coincides with the cover from $\mathcal U^{max}_X$. Since any cover of $\mathcal U^{max}_X$ is extended to the cover of  $b_m X$ ($\{\st_{\sigma}x\ |\ x\in X\}$ is a trace of the cover $$\{(\gets, (x_1, 0)),\ ((x_1, -1), (x_1, 1)),\ ((x_1, 0), (x_2, 0)),\ \ldots,\ ((x_{n-1}, 0), (x_n, 0)),\ ((x_n, -1), (x_n, 1)),\ ((x_n, 0), \to)\}{\rm)},$$ 
$b_m X=\beta_G X$ and the action $\theta_{\partial}$ is extended to the action $\tilde\theta_{\partial}: G\times b_m X\to b_m X$ (for  $g\in G$ 
\begin{itemize}
\item[] $\tilde\theta_{\partial} (g, (x, j))=(\theta_{\partial} (g, x), j)$, $x\in X$, $j=-1, 0, 1$;
\item[] $\tilde\theta_{\partial} (g, (A, B))=(\theta_{\partial} (g, A), \theta_{\partial} (g, B))$, $(A, B)\in\Gamma$).
\end{itemize}
It is easy to check that the action  $\tilde\theta_{\partial}$ is correctly defined, continuous, its restriction to $X=X^0$ coincides with $\theta_{\partial}$ and $X^+$, $X^0$, $X^-$,  $\Gamma$ are invariant subsets. 
\end{proof}

$c_m X=X\cup\Gamma$, $\Gamma$ is the set of gaps. The subsets $\inf$, $\sup$, $\Gamma$ and $\Gamma^0=\Gamma\setminus\{\inf, \sup\}$ are invariant subsets of $c_m X$.  

Any other $G$-compactification of $X$ is an image of $c_m X$ under the $G$-map $\varphi$ of compactifications. If one identifies points $\{\inf\}$ and $\{\sup\}$, then the compactification $c X$ is obtained ($\infty$ is the image of $\{\inf\}$ and $\{\sup\}$). If $X$ is a continuously ordered chain ($\Gamma^0=\emptyset$), then $c X=\alpha X=X\cup\{\infty\}$.

If $\xi\in\Gamma^0$ and $\varphi (\xi)=\varphi (\nu)$ ($\xi\sim\nu$), $\nu\ne\xi$, $\nu\in\Gamma^0$. Then the set $\st_{\xi}\nu=\{g\nu\ |\ g\in \st_{\xi}\}$ is dense either in $(\xi, \to)$ (if $\nu>\xi$) or in $(\gets, \xi)$ (if $\nu<\xi$). Since the equivalence class $[\xi]$ of $\xi$ is a closed and invariant subset of $c_m X$, $\sup\in [\xi]$ in the case $\nu>\xi$. Further, $\st_{\sup}\nu$ is dense in $c_m X$ and, hence, $[\nu]=[\xi]=c_m X$. The case $\nu<\xi$, $\nu\in\Gamma^0$ is examined similarly. 

If $\xi\in\Gamma^0$ and $\varphi (\xi)=\varphi (\sup)$, then $\st_{\sup}\xi$ is dense in $c_m X$ and, hence, $[\nu]=[\xi]=c_m X$. The case $\varphi (\xi)=\varphi (\inf)$ is examined similarly. From this it follows. 

\begin{lem}\label{compLOTS}
Let $X$ be an ultrahomogeneous {\rm LOTS}, $G=(\aut X, \tau_p)$.
There are two $G$-compactifications of $X$:  $c_m X>c X$. The map $c_m X\to c X$ of $G$-compactifications is elementary. 
\end{lem}

$b_m X=X^+\cup X^0\cup X^-\cup\Gamma$. The subsets $\inf$, $\sup$, $X^+$, $X^-$, $\Gamma$ and $\Gamma^0=\Gamma\setminus\{\inf, \sup\}$ are invariant subsets of $b_m X$.  

Any $G$-compactification of $X$ is an image of $b_m X$ under the $G$-map $\varphi$ of compactifications.  If one identifies points $\{\inf\}$ and $\{\sup\}$, then the $G$-compactification $b^{\leftrightarrow}X=b_m X/_{\inf\sim\sup}$ of $X$ is obtained ($\infty$ is the image of $\{\inf\}$ and $\{\sup\}$). 

If $\xi\in\Gamma^0$ and $\varphi (\xi)=\varphi (\nu)$ ($\xi\sim\nu$),  $\nu\ne\xi$, $\nu\in\Gamma^0$, then $\varphi (t)=\varphi (s)$ for $t, s\in b_m X\setminus X$. Indeed, $\st_{\xi}\nu=\{g\nu\ |\ g\in \st_{\xi}\}$ is dense either in $(\xi, \to)\setminus X^0$ (if $\nu>\xi$) or in $(\to, \xi)\setminus X^0$ (if $\nu<\xi$). Since the equivalence class $[\xi]$ of $\xi$ isa  closed and invariant subset of $b_m X$, $\sup\in [\xi]$ in the case $\nu>\xi$. Further, $\st_{\sup}\nu$ is dense in $b_m X\setminus X^0$  and, hence, $[\nu]=[\xi]=b_m X\setminus X^0$ and  $\varphi (t)=\varphi (s)$ for $t, s\in b_m X\setminus X^0$.  The Alexandroff one-point compactification $\alpha X$ is obtained. The case when $\xi>\nu$ is treated analogously.  

If $\xi\in\Gamma^0$ and $\varphi (\xi)=\varphi (t)$, where $t=(x, \pm 1)$, $x\in X$, then $\varphi (t)=\varphi (s)$ for $t, s\in b_m X\setminus X^0$. Indeed, if $\xi<t$, $\sup\in {\rm cl}\ {\rm St}_{\xi}t$ and $\sup\in [\xi]$. The reasoning, as in the previous case, shows that the Alexandroff one-point compactification $\alpha X$ is obtained. The case when $\xi>t$ is treated analogously.  

If $\xi\in\Gamma^0$ and $\varphi (\xi)=\varphi (\inf)$, or $\varphi (\xi)=\varphi (\sup)$, then the reasonings, as in the previous cases, show that the Alexandroff one-point compactification $\alpha X$ is obtained.

If $\varphi (x, 1)=\varphi (y, 1)$, $x<y\in X$, then ultratransitivity of the action yields that $\varphi (x, 1)=\varphi (z, 1)$ for any $z\in X$. Since $\varphi$ is continuous and the set $\{(z, 1)\ |\ z\in X\}$ is dense in $b_m X\setminus X^0$  $\varphi(x, 1)=\varphi (b_m X\setminus X^0)$ and the Alexandroff one-point compactification $\alpha X$ is obtained. The case when $x>y$ is treated analogously.  

If $\varphi (x, 1)=\varphi (\inf)$, or $\varphi (x, 1)=\varphi (\sup)$, then the reasonings, as in the previous cases, show that the Alexandroff one-point compactification $\alpha X$ is obtained. The case  $\varphi (x, -1)=\varphi (\inf)$, or $\varphi (x, -1)=\varphi (\sup)$ is treated similarly.

If $\varphi (x, -1)=\varphi (y, 1)$, $x<y\in X$, then also  $\varphi(x, 1)=\varphi (b_m X\setminus X^0)$ and the Alexandroff one-point compactification $\alpha X$ is obtained. The case when $x>y$ is treated analogously.  

If $\varphi (x, -1)=\varphi (x, 1)$, $x\in X$, then (due to the equivariance) for the pairs $(y, -1)$ and $(y, 1)$ $\varphi (y, -1)=\varphi (y, 1)$, $y\in X$, and one obtains the compactification $b^{\updownarrow}X=b_m X/_{(x, -1)\sim (x, 1),\ x\in X}$.

If  $\varphi (x, -1)=\varphi (x, 1)$, $x\in X$, and  $\varphi (\inf)=\varphi (\sup)$, then one obtains the compactification $b^{+}X=b_m X/_{(x, -1)\sim (x, 1),\ x\in X,\ \inf\sim\sup}$. From these observations, it follows. 

\begin{lem}\label{compd}
Let $X$ be an ultrahomogeneous discrete chain, $G=(\aut X, \tau_{\partial})$.
There are five $G$-compactifications of $X$:  
$$\begin{array}{ccclccc}
& & b^{\leftrightarrow}X=  &  & &  & \\
& \nearrow & =b_m X/_{\inf\sim\sup} &  \searrow    & & & \\
b_m X &   &   &  & b^{+}X=  & \rightarrow  &  \alpha X \\
 & \searrow & b^{\updownarrow}X= &  \nearrow  & =b_m X/_{(x, -1)\sim (x, 1),\ x\in X,\ \inf\sim\sup}  & & \\
& & =b_m X/_{(x, -1)\sim (x, 1),\ x\in X}  &  & &  & 
\end{array}$$

The maps $b_m X\to b^{\leftrightarrow}X$, $b^{\updownarrow}X\to  b^{+}X$ and $b^{+}X\  (b_m X,\ b^{\leftrightarrow}X,\  b^{\updownarrow}X)\to \alpha X$  of compactifications are elementary. 
\end{lem}

\begin{rem}
{\rm If $X$ is an ultrahomogeneous countable set (linearly isomorphic to $\mathbb Q$), then $G$-compactifications of $X$ are metrizable. 

Indeed, $c_m X=[0, 1]$ and $c X=S^1$. Compact space $b_m X$ has a countable base of intervals with endpoints in $X^0=X$ and all the other $G$-compactifications are its perfect image.}
\end{rem}

\subsection{Ellis compactifications of the automorphism group of ultrahomogeneous LOTS}\label{ultLOT}

\begin{thm}\label{Roelcke precomp4-3-1} Let $X$ be an ultrahomogeneous LOTS, $G=(\aut(X), \tau_p)$. Then 
\begin{itemize}
\item[{\rm (a)}]  The Ellis compactification $e_{c_m X} G$ is the completion of $(G, R_{\Sigma_{c_m X}})$, 
\item[{\rm (b)}]  $G$ is Roelcke precompact,  $b_r G<e_{c_m X} G$, 
\item[{\rm (c)}]  $e_{c_m X} G$ is the set of monotone self-maps $f$ of $c_m X$ {\rm(}$x<y\Longrightarrow f(x)\leq f(y)${\rm)} in the topology of pointwise convergence such that $f(\inf)=\inf$,  $f(\sup)=\sup$. 
\end{itemize}
If $X$ is a continuously ordered chain, then $e_{c_m X} G$  is the completion of $(G, R_{\Sigma_X})$. 
\end{thm}

\begin{proof}
(a) By Proposition~\ref{lemincl} $\mathcal V_G (c_m X)\subset R_{\Sigma_{c_m X}}$. To prove the equality $\mathcal V_G  (c_m X)=R_{\Sigma_{c_m X}}$, we use condition $(\star\star)$ from Proposition~\ref{a2-2}.

Let $\sigma=\{x_1, \ldots, x_n\}\in \Sigma_{c_m X}$, $\inf X<x_1<\ldots<x_n<\sup X$, (since $\st_{\inf}=\st_{\sup}=G$); the entourage ${\rm U}$ from the uniformity on $c_m X$ corresponds to the cover 
$$u_O=\{(\gets, y_2),\ (y_1, y_2),\ (y_1, y_4),\ \ldots,\ (y_{2m-3}, y_{2m}),\ (y_{2m-1}, y_{2m}),\ (y_{2m-1}, \to)\},$$
for  $y_1<t_1<y_2<\ldots<y_{2m-1}<t_m<y_{2m}\in X$, $m\in\mathbb N$,  $O=\bigcap\limits_{k=1}^m  [t_k, (y_{2k-1}, y_{2k})]$; and $g, h\in G$ are such that $(g(x), h(x))\in{\rm U}$, $x\in\sigma$. 

The balls $B(x, {\rm U})=\St (x, u_O)$, $x\in c_m X$, satisfy the following properties:
\begin{itemize}
\item if $y\in B(x, {\rm U})$, $y<x$, then exists $z<y$,  $z\in B(x, {\rm U})$ (if $y\in B(x, {\rm U})$, $x<y$, then exists $y<z$,  $z\in B(x, {\rm U})$);
\item if $y\in B(x, {\rm U})$, $y<x$ and $y<z_1<z_2<x$, then  $z_1\in B(z_2, {\rm U})$.
\end{itemize}
The first property follows from the openness of the balls, and the second one is valid since the elements of the cover $u_O$ are intervals. 

Let us construct an automorphism $g'\in g{\rm St}_{\sigma}$ such that $(h(x), g'(x))=(\tilde\theta_p (h, x), \tilde\theta_p (g', x))\in {\rm U}$ for any point $x\in c_m X$ and apply Remark~\ref{Rlkcomp}. 

The construction of the automorphism $g'$ is carried out on the closed intervals $$[\inf, x_1]=\{x\in c_m X\ |\ x\leq x_1\}, [x_1, x_2]=\{x\in c_m X\ |\ x_1\leq x\leq x_2\},\ldots, [x_{n-1}, x_n], [x_n, \sup].$$ 

(i) On $[\inf, x_1]$ in the case of $h(x_1)\geq g(x_1)$. If $(\inf, h(x_1))\in{\rm U}$, then $g'(x)=g(x)$ for $x\in [\inf, x_1]$. 

Otherwise, take $a_1<g(x_1)$, $a_1\in X$, $(a_1, h(x_1))\in{\rm U}$. We assume $x_1^-=h^{-1}(a_1)$ and take $y_1^-\in X$, $x_1^-<y_1^-<x_1$, such that $g(y_1^-)>a_1$. From the ultrahomogeneity of $X$, there exists $\varphi_1\in G$ such that $\varphi_1(x_1^-)=a_1$, $\varphi_1(y_1^-)=g(y_1^-)$. 

In the case of $h(x_1)\leq g(x_1)$.  If $(\inf, g(x_1))\in {\rm U}$, then $g'(x)=g(x)$ for $x\in [\inf, x_1]$. 

Otherwise, take $a_1<h(x_1)$, $a_1\in X$, $(a_1, g(x_1))\in {\rm U}$. We assume $x_1^-=h^{-1}(a_1)$ and take $y_1^-\in X$, $x_1^-<y_1^-<x_1$, such that $g(y_1^-)>a_1$. From the ultrahomogeneity of $X$, there exists $\varphi_1\in G$ such that $\varphi_1(x_1^-)=a_1$, $\varphi_1(y^-_1)=g(y_1^-)$. In both cases, when $g'$ is not yet defined, let us put
$$g'(x)=\left\{\begin{array}{lcl}
h(x) & \mbox{if} & x\in [\inf, x_1^-], \\
\varphi_1(x) & \mbox{if} & x\in [x_1^-, y_1^-].\\
g(x) & \mbox{if} & x\in [y_1^-, x_1].\\
\end{array}\right.$$ 
By the properties of balls $(h(x), g'(x))\in {\rm U}$, $x\in [\inf, x_1]$.

On the interval $[x_n, \sup]$ the construction is similar. 

\medskip

(ii) On $[x_k, x_{k+1}]$, $k=1, \dots, n-1$. The case of $h(x_k)\geq g(k_k)$,  $h(k_{k+1})\geq g(x_{k+1})$. If $(g(x_k), h(x_{k+1}))\in {\rm U}$, then $g'(x)=g(x)$ for $x\in [x_k, x_{k+1}]$.

Otherwise, take $b_k>h(x_k)$, $b_k\in X$, $(b_k, g(x_k))\in {\rm U}$, $a_{k+1}<g(x_{k+1})$, $a_{k+1}\in X$, $(a_{k+1}, h(x_{k+1}))\in {\rm U}$ ($b_k<a_{k+1}$). We assume $x_k^+=h^{-1}(b_k)$, $x_{k+1}^-=h^{-1}(a_{k+1})$, take $y_k^+\in X$, $x_k<y_k^+<x_k^+$, such that $g(y_k^+)< b_k$ and take $y_{k+1}^-\in X$, $x_{k+1}^-<y_{k+1}^-<x_{k+1}$, such that $g(y_{k+1}^-)> a_{k+1}$. From the ultrahomogeneity of $X$, there exists $\varphi_k^-\in G$ such that $\varphi_k^-(y_k^+)=g(y_k^+)$, $\varphi_k^-(x_k^+)=b_k$, there exists $\varphi_k^+\in G$ such that $\varphi_k^+(x_{k+1}^-)=a_{k+1}$, $\varphi_k^+(y_{k+1}^-)=g(y_{k+1}^-)$.

In the case of $h(x_k)\leq g(x_k)$,  $h(x_{k+1})\leq g(x_{k+1})$  the construction of $\varphi_k^-$ and $\varphi_k^+$ is similar. 

In the case of $h(x_k)\geq g(x_k)$,  $h(x_{k+1})\leq g(x_{k+1})$.   If  $(g(x_k), g(x_{k+1}))\in {\rm U}$, then $g'(x)=g(x)$ for $x\in [x_k, x_{k+1}]$.

Otherwise, take $b_k> h(x_k)$, $b_k\in X$, $(b_k, g(x_k))\in {\rm U}$, $a_{k+1}<h(x_{k+1})$, $a_{k+1}\in X$, $(a_{k+1}, g(x_{k+1}))\in {\rm U}$ ($b_k<a_{k+1}$). We assume $x_k^+=h^{-1}(b_k)$, $x_{k+1}^-=h^{-1}(a_{k+1})$, take $y_k^+\in X$, $x_k<y_k^+<x_k^+$, such that $g(y_k^+)< b_k$ and take $y_{k+1}^-\in X$, $x_{k+1}^-<y_{k+1}^-<x_{k+1}$, such that $g(y_{k+1}^-)> a_{k+1}$. From the ultrahomogeneity of $X$, there exists $\varphi_k^-\in G$ such that $\varphi_k^-(y_k^+)=g(y_k^+)$, $\varphi_k^-(x_k^+)=b_k$, there exists $\varphi_k^+\in G$ such that $\varphi_k^+(x_{k+1}^-)=a_{k+1}$, $\varphi_k^+(y_{k+1}^-)=g(y_{k+1}^-)$. 

In the case of $h(x_k)\leq g(x_k)$, $h(x_{k+1})\geq g(x_{k+1})$  the construction of $\varphi_k^-$ and $\varphi_k^+$ is similar. 

In all the cases, when $g'$ is not yet defined, let us put
$$g'(x)=\left\{\begin{array}{lcl}
g(x) & \mbox{if} & x\in [x_k, y_k^+], \\
\varphi_k^-(x) & \mbox{if} & x\in [y_k^+, x_k^+], \\
h(x) & \mbox{if} & x\in [x_k^+, x_{k+1}^-], \\
\varphi_k^+(x) & \mbox{if} & x\in [x_{k+1}^-, y_{k+1}^-],\\
g(x) & \mbox{if} & x\in [y_{k+1}^-, x_{k+1}],\\
\end{array}\right.$$ 
It follows from the construction and the properties of balls that $(h(x), g'(x))\in {\rm U}$ for any point $x\in c_m X$. Hence, $e_{c_m X} G$ is the completion of $(G, R_{\Sigma_{c_m X}})$. 

\medskip

(b) By Remark~\ref{a2-1} $L\wedge R\subset R_{\Sigma_X}$, $R_{\Sigma_X}\subset R_{\Sigma_{c_m X}}=\mathcal V_G (c_m X)$ and $\mathcal V_G (c_m X)$ is totally bounded. Hence, $G$ is Roelcke precompact and $b_r G\leq e_{c_m X} G$.

Take $x\in X$ and entourage ${\rm U}$ from the uniformity on $c_m X$ such that $\exists\ y, z\in X$, $y<z$ and $(y, z)\not\in {\rm U}$. Fix ${\rm W}=\{(h', g')\in G\times G\ |\ (g'(x), h'(x))\in {\rm U}\}\in\mathcal V_G (c_m X)$. Ultrahomogeneity of $X$ yields that $\forall\ O\in N_G (e)$ $\exists\ g\in G$ and $\exists\ h\in O$ such that $g(x)=y$, $x<h(x)$ and $g(h(x))=z$. Therefore, $gh\in OgO$, however, $(g(x), (gh)(x))\not\in {\rm U}$. Hence,  $b_r G<e_{c_m X} G$.

\medskip

(c) Let $f\in e_{c_m X} G$, $x, y\in c_m X$, $x<y$. If $f(x)>f(y)$, then there are nbds $O_{f(x)}$ and $O_{f(y)}$ of $f(x)$ and $f(y)$ respectively, such that $z\in O_{f(x)},\ z'\in O_{f(y)}\Longrightarrow z>z'$ (designation $O_{f(x)}>O_{f(y)}$). Then $W=[x, O_{f(x)}]\cap [y, O_{f(y)}]$ is a nbd of $\imath_{c_m X}(f)$ and  $W\cap \imath_{c_m X} (G)=\emptyset$. Hence, $f\not\in e_{c_m X} G$. 

If $f(\inf)=y\ne \inf$, then for a nbd $O_{y}$ of $y$ such that $\inf\not\in O_{y}$ one has $W=[y, O_{y}]\cap \imath_{c_m X}(G)=\emptyset$ and, hence, $f\not\in e_{c_m X} G$. The condition $f(\sup)=\sup$ is verified similarly.

Let $f\in (c_m X)^{c_m X}$ be monotone, $f(\inf)=\inf$, $f(\sup)=\sup$, and $W=\bigcap\limits_{k=1}^n [x_k, O_{f(x_k)}]$ is its arbitrary nbd. Without loss of generality, one can consider that $x_1=\inf$, $x_2,\ldots, x_{n-1}\in X$, $x_n=\sup$, $x_1<\ldots<x_n$ and $O_{f(x_{k+1})}=O_{f(x_{k})}$ if $f(x_{k+1})=f(x_{k})$, and $O_{f(x_{k+1})}>O_{f(x_{k})}$ (i.e. that $\forall\ z\in O_{f(x_{k+1})}$ and $\forall\ z'\in O_{f(x_{k})}$ $\Longrightarrow$ $z>z'$)  otherwise, $k=1, \ldots, n-1$.

There are points $y_2,\ldots, y_{n-1}\in X$ such that $\inf<y_2<\ldots<y_{n-1}<\sup$ and $y_k\in O_{f(x_{k})}$, $k=2, \ldots, n-1$. Ultrahomogeneity of $X$  yields that the exists $g\in G$ such that $g(x_k)=y_k$, $k=2, \ldots, n-1$. Evidently, $\imath_{c_m X} (g)\in W$. Hence, $f\in e_{c_m X} G$.

\medskip

If $X$ is a continuously ordered chain, then $R_{\Sigma_X}=R_{\Sigma_{c_m X}}$ ($c_m X=X\cup\{\inf, \sup\}$, $\st_{\inf}=\st_{\sup}=G$) and $e_{c_m X} G$  is the completion of $(G, R_{\Sigma_X})$.
\end{proof}

\begin{que}
Is the equality $\mathfrak{E} (b_r G)=e_{c_m X} G$ valid?
\end{que}


For the map  $c_m X\to c X$  of $G$-compactifications of $X$ 
$$I=\{f\in e_{c_m X} G\ |\ f(x)=\inf\,\vee\,\sup,\ x\in c_m X\}$$
is a nonempty compact ideal of $e_{c_m X} G$. Since $e_{c_m X} G$ is the set of monotone self-maps on $c_m X$, it is easy to verify that the condition (EF) from Proposition~\ref{Reesquot} holds. 

\begin{cor}
$$e_{c X} G=e_{c_m X} G/I.$$
\end{cor}

\begin{pro}\label{LOTcomp}
Let $X$ be an ultrahomogeneous LOTS, $G=(\aut(X), \tau_p)$. Then  
\begin{itemize}
\item[{\rm (a)}] The Ellis compactification $e_{c X} G$ is the set of all self-maps $f$ of $c X$ such that  $f(\infty)=\infty$,  the restriction of $f$ to $f^{-1} (c X\setminus\{\infty\})$ is a monotone map of $f^{-1} (c X\setminus\{\infty\})$ to $c X\setminus\{\infty\}$ {\rm(}with the linear order induced on $c X\setminus\{\infty\}$ as a subset of $c_m X${\rm)}, 
\item[{\rm (b)}] $e_{c X} G$ and $b_r G$ are incomparable compactifications of $G$.
\end{itemize}
\end{pro}

\begin{proof}
(a) Let $f\in e_{c X} G$, $Y=c X\setminus\{\infty\}$, $x, y\in f^{-1} Y$, $x<y$. If $f(x)>f(y)$, then there are nbds $U_{f(x)}$ and $U_{f(y)}\subset Y$ such that $U_{f(x)}>U_{f(y)}$. Then $W=[x, U_{f(x)}]\cap [y, U_{f(y)}]$ is a nbd of $\imath_{c X} (f)$ and  $W\cap\imath_{c X} (G)=\emptyset$. Hence, $f\not\in e_{c  X} G$.

Let $f\in e_{c X} G$. If $f(\infty)\in Y$, then take a nbd $U$ of $f(y)$ such that $\infty\not\in U$. Then $W=[\infty, U]$ is a nbd of $\imath_{c X} (f)$ and $W\cap\imath_{c X} (G)=\emptyset$. Hence, $f\not\in e_{c X} G$.

Let $f\in (c X)^{c X}$ satisfies the conditions of (a) and $W=\bigcap\limits_{k=1}^n [x_k, U_{f(x_k)}]$ is its arbitrary nbd. Without loss of generality, one can consider that $x_1,\ldots, x_n\in X\subset Y$, $x_1<\ldots<x_n$; $U_{f(x_{k})}=U_{f(x_{k+1})}\subset Y$ if $f(x_{k})=f(x_{k+1})\in Y$;  $U_{f(x_{k})},\ U_{f(x_{k+1})}\subset Y$, $U_{f(x_{k})}<U_{f(x_{k+1})}$ if $f(x_{k})<f(x_{k+1})\in Y$; and if $f(x_{j_1})=\ldots=f(x_{j_m})=\infty$, then 
$U_{\infty}$ is the nbd of $\infty$ which is disjoint from other nbds $U_{f(x_{k})}$ when $f(x_k)\in Y$. 

There are points $y_1,\ldots, y_n\in X$ such that $y_1<\ldots<y_n$ and $y_k\in U_{f(x_{k})}$, $k=1, \ldots, n$. Ultrahomogeneity of $X$ yields that the exists $g\in G$ such that $g(x_k)=y_k$, $k=1, \ldots, n$. Evidently, $\imath_{c X} (g)\in W$. Hence,   $f\in e_{c X} G$. 

(b) Take $x\in X$ and entourage ${\rm U}$ from the uniformity on $cX$ such that $\exists\ y, z\in X$, $y<z$, $(y, z)\not\in {\rm U}$. Fix ${\rm W}=\{(h', g')\in G\times G\ |\ (g'(x), h'(x))\in {\rm U}\}\in\mathcal V_G (c X)$. Ultrahomogeneity of $X$ yields that $\forall\ O\in N_G (e)$ $\exists\ g\in G$ and $\exists\ h\in O$ such that $x<h(x)$, $g(x)=y$, $g(h(x))=z$. Therefore, $gh\in OgO$, however, $(g(x), (gh)(x))\not\in {\rm U}$. Hence, $\mathcal V_G (c X)\not\subset L\wedge R$ and $b_r G\ngeq e_{c X} G$.

Let $O\in  N_G (e)$ be such that $O\subset [x, (a, b)]$, $O=O^{-1}$, where $\inf<a<x<b<\sup$, $a, x, b\in X$. An arbitrary entourage from  $\mathcal V_G (c X)$ is defined by an entourage ${\rm U}$ from the uniformity on $cX$ and a finite number of points $x_1<\ldots<x_n$ in $c X\setminus\{\infty\}$. Then for the map $\varphi: c_m X\to c X$ of compactifications $\varphi^{-1} (B(\infty, {\rm U}))$ may be assumed to consist of two disjoint intervals $\inf\in V_{\inf}$ and $\sup\in V_{\sup}$ in $c_m X$. 

Take $g\in G$ such that $g (a)\in V_{\sup}$, $g(a)>b$, and $h\in G$ such that $h (b)\in V_{\inf}$, $h(b)<a$, and $(g, h)\in {\rm U}$. The last condition is realized since $g (\min\{a, x_1\})\in V_{\sup}$ and $h (\max\{b, x_n\})\in V_{\inf}$ is a sufficient demand, which doesn't contradict the initial limitations. 

If $h\in OgO$, then $\xi h=g\psi$ for some $\xi, \psi\in O$. However, $h(x)<a$, $(\xi h)(x)< b$, $\psi (x)>a$, $(g \psi) (x)>b$ and, hence, $\xi h\ne g\psi$. The obtained contradiction shows that $\mathcal V_G (c X)\not\supset L\wedge R$ and $b_r G\nleq e_{c X} G$.
\end{proof}

Since $e_{c_m X} G>b_r G$ and $e_{c X} G$ and $b_r G$ are incomparable compactifications of $G$, from~\cite[Remark 3.14]{KozlovLeiderman} one has. 

\begin{cor}\label{LOTSincom}
$e_{c_m X} G$  and $e_{c X} G$ are not sm-compactifications of $G$.
\end{cor}

\begin{rem}\label{Helly}
{\rm If $c_m X=[0, 1]$, then $e_{c_m X} G$ is a retract of the Helly space~\cite[Exercise 3.2.E]{Engelking}. Indeed, the restriction of the projection $[0, 1]^{[0, 1]}\to [0, 1]^{(0, 1)}$ to the Helly space is the required rectraction. Therefore, $e_{c_m X} G$ is a first-countable, separable non-metrizable compact space (assuming that $e_{c_m X} G$ is metrizable, the diagonal product of the restrictions of the projections $[0, 1]^{[0, 1]}\to [0, 1]^{(0, 1)}$,  $[0, 1]^{[0, 1]}\to [0, 1]^{\{0, 1\}}$ to the Helly space will be embedding of the Helly space into a metrizable space, which is a contradiction with non-metrizability of the Helly space).}
\end{rem}

\begin{que} 
Does there exist sm-compactification of $\aut (X)$, where $X$ is an ultrahomogeneous LOTS?
\end{que}

\subsection{Ellis compactifications of the automorphism group of ultrahomogeneous discrete chain}\label{ultLOTa}

\begin{thm}\label{Roelcke precomp4-2-1} Let $X$ be an ultrahomogeneous discrete chain, $G=(\aut(X), \tau_{\partial})$. Then 
\begin{itemize}
\item[{\rm (a)}]  Ellis compactification $e_{b_m X} G$ is the completion of $(G, R_{\Sigma_{b_m X}})$, 
\item[{\rm (b)}]  $G$ is Roelcke precompact, $b_r G\leq e_{b_m X} G$ and ${\mathfrak E}(b_r G)=e_{b_m X} G$ {\rm(}see {\rm Remark~\ref{compmapsth})}, 
\item[{\rm (c)}]  $e_{b_m X} G$ is the set of self-maps $f$ of $b_m X$  in the topology of pointwise convergence such that  
\begin{itemize}
\item[{\rm (i)}] $f$ is monotone {\rm (}if $x<y$, then $f(x)\leq f(y)${\rm)}, 
\item[{\rm (ii)}] $f(x)\not\in X$ if $x\in b_m X\setminus X$, 
\item[{\rm (iii)}] if $f(x)=f(y)$,  $x\ne y$, then $f(x)\in b_m X\setminus X$, 
\item[{\rm (iv)}] either $f((x, -1))=f((x, 0))=f((x, 1))\not\in X$, or $f((x, -1))=(y, -1)$, $f((x, 0))=(y, 0)$,  $f((x, 1))=(y, 1)$  and 
\item[{\rm (v)}] $f(\inf)=\inf$,  $f(\sup)=\sup$, 
\end{itemize}
\item[{\rm (d)}] $b_r G$ is the set of maps $f$ of $X$ to $b_m X$  in the topology of pointwise convergence such that  {\rm (i')} $f$ is monotone,  {\rm (ii')} if $f(x)=f(y)$,  $x\ne y$, then $f(x)\in b_m X\setminus X$. 
\end{itemize}
If $X$ is a continuously ordered chain, then $b_r G=e_{b_m X} G$. 
\end{thm}

\begin{proof}  
(a) By Proposition~\ref{lemincl} $\mathcal V_G(b_m X)\subset R_{\Sigma_{b_m X}}$. To prove the equality $\mathcal V_G(b_m X)=R_{\Sigma_{b_m X}}$, we use condition $(\star\star)$ from Proposition~\ref{a2-2}.

For the action $\theta_{\partial}: G\times X\to X$ let $\tilde\theta_{\partial}: G\times b_m X\to b_m X$ be its extension. Since $\st_{(x, j)}$, $j=-1, 0, 1$, for the action $\tilde\theta_{\partial}$ coincides with $\st_x$ for the action $\theta_{\partial}$ and $\st_{\inf}=\st_{\sup}=G$, one can take arbitrary $\sigma=\{x_k\ |\ k=1,\ldots, n\}\in\Sigma_{b_m X}$, $x_1<x_2<\ldots< x_n$, such that $\sigma=\sigma_X\cup\sigma_{\Gamma}$, $\sigma_X\in\Sigma_{X^0}(=\Sigma_{X})$ $\sigma_{\Gamma}\in\Sigma_{\Gamma}\setminus\{\inf, \sup\}$. Let ${\rm U}={\rm U}\circ {\rm U}$ be the entourage from uniformity on $b_m X$ which corresponds to the uniform cover 
$$\{(\gets, (y_1, 0)),\ \{(y_1, 0)\},\ ((y_1, 0), (y_2, 0)),\ \{(y_2, 0)\},\ \ldots,\ ((y_{m-1}, 0), (y_m, 0)),\ \{(y_m, 0)\},\ ((y_m, 0), \to)\},$$
$\sigma'=\{y_k=(y_k, 0)\ |\ k=1,\ldots, m\}\in\Sigma_X$ (see the proof of Proposition~\ref{betaG}). 

Therefore, the condition $(\star\star)$ in Proposition~\ref{a2-2} can be reformulated in the following form 
$$\begin{array}{c}
\mbox{\rm if for}\ \sigma=\sigma_X\cup\sigma_{\Gamma}\in\Sigma_{b_m X},\ \sigma_X\subset\sigma'\in\Sigma_{X},\ \sigma_{\Gamma}\in\Sigma_{\Gamma}\setminus\{\inf, \sup\}\ \mbox{and}\ g, h\in G, \\
\mbox{the inclusion}\ (h(x), g(x))=(\tilde\theta_{\partial}(h, x), \tilde\theta_{\partial}(g, x))\in {\rm U}\ \mbox{\rm holds for}\ x\in\sigma, \\
\mbox{\rm then there exists}\ \ g'\in g\st_{\sigma}\ \ \mbox{\rm such that}\ (h(x), g'(x))=(\tilde\theta_{\partial}(h, x), \tilde\theta_{\partial}(g', x))\in {\rm U}\  \mbox{\rm holds for}\ x\in b_m X.\\ 
\end{array}$$

The construction of the automorphism $g'$ is carried out on the closed intervals 
$$[\inf, x_1], [x_1, x_2],\ldots, [x_{n-1}, x_n], [x_n, \sup].$$  

(I) The map $g'$ on the interval $[\inf, x_1]$. If $(\inf, h(x_1))\in {\rm U}$ (then $(\inf, g(x_1))\in {\rm U}$) or $h(x_1)=\inf \sigma'$ (then $g(x_1)=h(x_1)$, $x_1=(x_1, 0)$, $g((x_1, -1))=h((x_1, -1))$ and $(\inf, h((x_1, -1)))\in {\rm U}$), then 
$$g'(x)=g(x)\ \mbox{for}\ x\in [\inf, x_1].$$

Otherwise, in the case $g(x_1)\leq h(x_1)$. If $g(x_1)\not\in\sigma'$, then take  $a_1<g(x_1)$ such that $a_1\in X^0$,  $(a_1, h(x_1))\in {\rm U}$  (and hence, $(a_1, g(x_1))\in {\rm U}$). If $g(x_1)\in\sigma'$, then $x_1=(x_1, 0)$, $g(x_1)=h(x_1)$ and $g((x_1, -1))=h((x_1, -1))$. Take $a_1<g((x_1, -1))$ such that $a_1\in X^0$,  $(a_1, h((x_1, -1)))\in {\rm U}$   (and hence,  $(a_1, g((x_1, -1)))\in {\rm U}$). Put $t=x_1$ if $g(x_1)\not\in\sigma'$ and $t=(x_1, -1)$ if $g(x_1)\in\sigma'$.

We assume that $x_1^-=h^{-1}(a_1)$ and take $x_1^-<z_1^-<t$, $z_1^-\in X^0$, $a_1<g(z_1^-)$ (using continuity of $g$ and the inequality $a_1<g(t)$). Since $X$ is ultrahomogeneous, there exists $\varphi_1\in G$ such that $\varphi_1(x_1^-)=a_1=h(x_1^-)<\varphi_1(z_1^-)=g(z_1^-)$.

In the case $g(x_1)>h(x_1)$ take  $a_1<h(x_1)$ such that $a_1\in X^0$,  $(a_1, g(x_1))\in {\rm U}$  (and hence, $(a_1, h(x_1))\in {\rm U}$). We assume that $x_1^-=h^{-1}(a_1)$ and take $x_1^-<z_1^-<y_1$, $z_1^-\in X^0$, $a_1<g(z_1^-)$ (using continuity of $g$ and the inequality $a_1<g(y_1)$). Since $X$ is ultrahomogeneous, there exists $\varphi_1\in G$ such that $\varphi_1(x_1^-)=a_1=h(x_1^-)<\varphi_1(z_1^-)=g(z_1^-)$. Put 
$$g'(x)=\left\{\begin{array}{lcl}
h(x) & \mbox{if} & x\in [\inf, x_1^-], \\
\varphi_1(x) & \mbox{if} & x\in [x_1^-, z_1^-],\\
g(x) & \mbox{if} & x\in [z_1^-, x_1].\\
\end{array}\right.$$

On the interval $[x_n, \sup]$ the construction is similar. 

\medskip

(II)  The map $g'$ on on the interval $[x_k, x_{k+1}]$, $k=1, \dots, n-1$. If 
$$(h(x_k), h(x_{k+1}))\in {\rm U},\ \mbox{or}$$  
$h(x_k)\in\sigma'$, $h(x_{k+1})\in\sigma'$ (then $g(x_k)\in\sigma'$, $g(x_{k+1})\in\sigma'$, $x_k=(x_k, 0)$, $x_{k+1}=(x_{k+1}, 0)$, $h((x_k, 1))=g((x_k, 1))$, $h((x_{k+1}, -1))=g((x_{k+1}, -1))$) and $(h((x_k, 1)), h((x_{k+1}, -1)))\in {\rm U}$, or 

\noindent $h(x_k)\in\sigma'$, $h(x_{k+1})\not\in\sigma'$ and $(h((x_k, 1)), h(x_{k+1}))\in {\rm U}$, or 

\noindent $h(x_k)\not\in\sigma'$, $h(x_{k+1})\in\sigma'$ and $(h( x_k), h((x_{k+1}, -1)))\in {\rm U}$, then
$$g'(x)=g(x)\ \mbox{for}\ x\in [x_k, x_{k+1}].$$  

Otherwise, let us examine the case $h(x_k)\geq g(x_k)$, $h(x_{k+1})\geq g(x_{k+1})$. 

(a) If $h(x_k), h(x_{k+1})\not\in\sigma'$ (and, hence, $g(x_k), g(x_{k+1})\not\in\sigma'$), then take $b_k>h(x_k)$, $b_k\in X^0$, $(b_k, g(x_k))\in {\rm U}$ and $a_{k+1}<g(x_{k+1})$, $a_{k+1}\in X^0$, $(a_{k+1}, h(x_{k+1}))\in {\rm U}$. $b_k<a_{k+1}$.  Put $t=x_k$ and $\tau=x_{k+1}$. 

(b) If $h(x_k)$, $h(x_{k+1})\in\sigma'$,  then take $b_k>h((x_k, 1))$, $b_k\in X^0$, $(b_k, h((x_k, 1)))\in {\rm U}$ and $a_{k+1}<g((x_{k+1}, -1))$, $a_{k+1}\in X^0$, $(a_{k+1}, g((x_{k+1}, -1)))\in {\rm U}$. $b_k<a_{k+1}$.  Put $t=(x_k, 1)$ and $\tau=(x_{k+1}, -1)$. 

(c) If $h(x_k)\in\sigma'$, $h(x_{k+1})\not\in\sigma'$, then take $b_k>h((x_k, 1))$, $b_k\in X^0$, $(b_k, h((x_k, 1)))\in {\rm U}$ and $a_{k+1}<g(x_{k+1})$, $a_{k+1}\in X^0$, $(a_{k+1}, h(x_{k+1}))\in {\rm U}$. $b_k<a_{k+1}$.  Put $t=(x_k, 1)$ and $\tau=x_{k+1}$. 

(d) If $h(x_k)\not\in\sigma'$, $h(x_{k+1})\in\sigma'$, take $b_k>h(x_k)$, $b_k\in X^0$, $(b_k, g(x_k))\in {\rm U}$ and $a_{k+1}<g((x_{k+1}, -1))$, $a_{k+1}\in X^0$, $(a_{k+1}, g((x_{k+1}, -1)))\in {\rm U}$. $b_k<a_{k+1}$.  Put $t=x_k$ and $\tau=(x_{k+1}, -1)$. 

We assume that $x_k^+=h^{-1}(b_k)$, $x_{k+1}^-=h^{-1}(a_{k+1})$. From the continuity of $g$ and the inequality $g(t)<b_k$ there exists $t<z^+_k<x_k^+$, $z^+_k\in X^0$, such that $g(z_k^+)<b_k$, and from the inequality $a_{k+1}<g(\tau)$ there exists $x_{k+1}^-<z^-_{k+1}<\tau$, $z^-_{k+1}\in X^0$, such that $a_{k+1}<g(z_{k+1}^-)$. Since $X$ is ultrahomogeneous, there exist $\varphi_k^-\in G$ such that $\varphi_k^-(z^+_k)=g(z_k^+)<\varphi_k^-(x_k^+)=b_k=h(x_k^+)$ and 
$\varphi_k^+\in G$ such that $\varphi_k^+(x_{k+1}^-)=a_{k+1}=h(x_{k+1}^-)<\varphi_k^+(z^-_{k+1})=g(z_{k+1}^-)$. 
Put
$$g'(x)=\left\{\begin{array}{lcl}
g( x) & \mbox{if} & x\in [x_k, z_k^+], \\
\varphi_k^-(x) & \mbox{if} & x\in [z_k^+, y_k^+], \\
h(x) & \mbox{if} & x\in [y_k^+, y_{k+1}^-], \\
\varphi_k^+(x) & \mbox{if} & x\in [x_{k+1}^-, z_{k+1}^-], \\
g(x) & \mbox{if} & x\in [z_{k+1}^-, x_{k+1}]. \\
\end{array}\right.$$ 

In the cases $h(x_k)\leq g(x_k)$,  $h(x_{k+1})\leq g(x_{k+1})$;  $h(x_k)\geq g(x_k)$,  $h(x_{k+1})\leq g(x_{k+1})$; and $h(x_k)\leq g(x_k)$, $h(x_{k+1})\geq g(x_{k+1})$ the considerations are similar. 

\medskip

It follows from the construction that $g'\in G$, $g'\in g\st_{\sigma}$ and $(h(x), g'(x))\in {\rm U}$ for any point $x\in b_m X$. Hence, $R_{\Sigma_{b_m X}}=\mathcal V_G (b_m X)$ and $e_{b_m X} G$ is the completion of $(G, R_{\Sigma_{b_m X}})$.

\medskip

(b) Since $R_{\Sigma_X}\subset R_{\Sigma_{b_m X}}=\mathcal V_G (b_m X)$ and  $R_{\Sigma_X}=L\wedge R$ by Remark~\ref{a2-1},  $L\wedge R\subset\mathcal V_G (b_m X)$. Hence, $G$ is Roelcke precompact and $b_r G\leq e_{b_m X} G$. From item (4) of Remark~\ref{propchain} and item (2) of Proposition~\ref{autultr} it follows ${\mathfrak E}(b_r G)=e_{b_m X} G$.

\medskip

(c) Fulfillment of conditions (i)--(v) for $f\in e_{b_m X} G$. 

(i) Let $x,\ y\in b_m X$, $x<y$ and $f(x)>f(y)$. Take disjoint nbds $O_{f(x)}$, $O_{f(y)}$ of $f(x)$ and $f(y)$ respectively, which are intervals. Then $O_f=[x,  O_{f(x)}]\cap [y,  O_{f(y)}]$ is a nbd of $f$, but $\imath_{b_m X} (G)\cap O_f=\emptyset$. Hence, $f\not\in\cl\,\imath_{b_m X} (G)=e_{b_m X} G$. 

(ii) Let $x\in b_m X\setminus X$ and $f(x)\in X$. Then  $O_f=[x, \{f(x)\}]$ is a nbd of $f$,  but $\imath_{b_m X} (G)\cap O_f=\emptyset$. Hence, $f\not\in e_{b_m X} G$. 

(iii) Let $f(x)=f(y)\in X$, $x,\ y\in b_m X$,  $x\ne y$. Then  $O_f=[x,  \{f(x)\}]\cap [y,  \{f(y)\}]$ is a nbd of $f$, but $\imath_{b_m X} (G)\cap O_f=\emptyset$.  Hence, $f\not\in e_{b_m X} G$.

(iv) If $f((x, -1))=f((x, 0))<f((x, 1))$, then $f((x, 0))$, $f((x, 1))\in b_m X\setminus X$ by (ii) and there exists $z\in X^0$ such that $f((x, 0))< z<f((x, 1))$. For the disjoint nbds $U=(\gets, z)$ and $V=(z, \to)$ of $f((x, 0))$ and $f((x, 1))$ respectively,  $O_f=[(x, 0),  U]\cap [(x, 1),  V]$ is a nbd of $f$, but $\imath_{b_m X} (G)\cap O_f=\emptyset$. 

The same argument works in the case $f((x, -1))<f((x, 0))=f((x, 1))$. Therefore, either $f((x, -1))=f((x, 0))=f((x, 1))\not\in X$ or $f((x, -1))<f((x, 0))<f((x, 1))$ and  there is no $z$ (by the above argument) such that $f((x, -1))<z<f((x, 0))$ or $f((x, 0)<z<f((x, 1))$.  Hence,  $f((x, -1))=(y, -1)$, $f((x, 0))=(y, 0)$,  $f((x, 1))=(y, 1)$. 

(v) Let $f(\inf)=x\ne\inf$. Take nbd $O_x$ of $x$, $\inf\not\in O_x$. Then $O_f=[\inf, O_x]$ is a nbd of $f$,  but  $\imath_{b_m X} (G)\cap O_f=\emptyset$. Hence, $f\not\in e_{b_m X} G$. The same argument works for $\sup$. 

\medskip

Let $f: b_m X\to b_m X$ satisfies conditions (i)--(v).  Take an arbitrary nbd $O_f=\bigcap\limits_{k=1}^n [x_k, O_{f(x_k)}]$ of $f$ 
where, without lose of generality, $\inf<x_1<\ldots<x_n<\sup$ (use (v)), and $x_k\in X^0\cup\Gamma$, $k=1,\ldots, n$, (use (iv)); 

\noindent $O_{f(x_k)}$, $k=1,\ldots, n$ are clopen intervals, $O_{f(x_k)}=O_{f(x_{k+1})}$ if $f(x_k)=f(x_{k+1})$ and  $O_{f(x_k)}\cap O_{f(x_{k+1})}=\emptyset$ if $f(x_k)<f(x_{k+1})$,  $k=1,\ldots, n-1$ (use (i)). 

Moreover, if $f(x_k)=(y_k, 0)$, then $O_{f(x_k)}=\{(y_k, 0)\}$, 

\noindent if $f(x_k)=(y_k, 1)$, then $O_{f(x_k)}=((y_k, 0), (z_k, 0))$ for some $z_k\in X$, 

\noindent if $f(x_k)=(y_k, -1)$, then $O_{f(x_k)}=((z_k, 0), (y_k, 0))$ for some $z_k\in X$.

In order to construct automorphism $g\in O_f$ we sequentially determine two finite increasing sequences of points $t_k\leq t'_k<t_{k+1}$ and $\tau_k\leq\tau'_k<\tau_{k+1}$ in $X$, $k=1,\ldots, n$. 

If $O_{f(x_k)}=\{(y_k, 0)\}$, then $x_k=(x_k, 0)\in X^0$ by (ii), $t'_k=t_k=x_k$, and $\tau_k=\tau'_k=y_k$.

Let $O_{f(x_k)}=((y_k, 0), (z_k, 0))$. 

If $x_k=(x_k, 0)\in X^0$, then take any $(\tau_k, 0)\in ((y_k, 0), (z_k, 0))$, $\tau'_{k-1}<\tau_k$ (if $k-1=0$ , then $\inf<(\tau_k, 0)$) and put $t_k=t'_k=x_k$, $\tau_k=\tau'_k$. 

If  $x_k\in\Gamma$, then take any $(\tau_k, 0),\ (\tau'_k, 0)\in ((y_k, 0), (z_k, 0))$, $\tau'_{k-1}<\tau_k<\tau'_k$ and any $t_k,\ t'_k\in X$ such that $t'_{k-1}<t_k<x_k<t'_{k}<x_{k+1}$ (if $k=1$, then $\inf<(t_1, 0)$; if $k=n$, then $(t'_n, 0)<\sup$).  

Since $X$ is ultrahomogeneous, there exists  $g\in G$ that maps an ordered sequence of $t_k,\ t'_k$ onto an ordered sequence of  $\tau_k,\ \tau'_k$ (coinciding points are identified). Evidently, $g\in O_f$ and, hence, $f\in e_{b_m X} G$.

\medskip

(d)  Let $\imath_X: G\to (b_m X)^X$ be the composition of embedding $\imath_{b_m X}: G\to (b_m X)^{b_m X}$ and projection $\pr: (b_m X)^{b_m X}\to (b_m X)^X$. $\imath_X$ is a topological embedding since homeomorphisms from $G$ are continuous and $X$ is a dense subset of $b_m X$. 

Let $\mathcal V_G^X (b_m X)$ be the subset uniformity on $\imath_X (G)$ as a subset of the compact space $(b_m X)^X$. As in item (a), it can be shown that $\mathcal V_G^X (b_m X)=R_{\Sigma_X}$. Since $L\wedge R=R_{\Sigma_X}$, the Roelcke compactification $b_r G$ is the closure of $\imath_X (G)$  in $(b_m X)^{X}$. The restriction $\pr|_{\cl\,\imath_{b_m X} (G)}$  of the projection to $\cl\,\imath_{b_m X} (G)$ is the map of the Ellis compactification $e_{b_m X} G$ onto $b_r G$. This implies the description of $b_r G$ as the set of maps $f$ of $X$ to $b_m X$  in the topology of pointwise convergence such that (i') $f$ is monotone,  (ii') if $f(x)=f(y)$, then $f(x)\in b_m X\setminus X$.

\medskip

If $X$ is a continuously ordered chain, then $R_{\Sigma_X}=R_{\Sigma_{b_m X}}$, because $\st_{\inf}={\rm St}_{\sup}=G$ and $\st_{(x, -1)}=\st_{(x, 1)}=\st_{(x, 0)}$ for $x\in X$. Hence, $\mathcal V_G^X (b_m X)=L\wedge R=R_{\Sigma_X}=R_{\Sigma_{b_m X}}=\mathcal V_G (b_m X)$ and  $b_r G=e_{b_m X} G$. 
\end{proof}

\begin{cor}\label{contchain}
Let $X$ be an ultrahomogeneous discrete chain, $G=(\aut(X), \tau_{\partial})$. Then 
\begin{itemize}
\item[{\rm(1)}] if $b_r G\not\in\mathbb{E}(G)$, then there are no $b G\in\mathbb{E}(G)$ such that $b_r G<b G<e_{b_m X} G$, 
\item[{\rm(2)}] $b_r G\in\mathbb{E}(G)$ iff $X$ is a continuously ordered chain.
\end{itemize}
\end{cor}

\begin{proof} (1) follows from Remark~\ref{nocomp}.

Sufficiency in (2) is proved in Theorem~\ref{Roelcke precomp4-2-1}. Necessity. Suppose that $X$ is not a continuously ordered chain. Let us show that  $R_{\Sigma_X}\ne R_{\Sigma_{b_m X}}$.

Take $y\in X$ and let the cover $\Omega_y=\{(\gets, (y, 0))\cup\{(y, 0)\}\cup((y, 0), \to)\}$ be from uniformity on $b_m X$, $\xi$ is a proper gap in $X$, $O\in N_G(e)$ is such that $\{Ox\ |\ x\in b_m X\}\succ\Omega_y$.

For $\sigma\in\Sigma_X$ let  $x, x'\in\sigma$ be such that $x<\xi<x'$, and $\sigma\cap (x, x')=\emptyset$ (without loss of generality, we may assume that this always holds). Take  $t_1, t_1', t_2, t_2'\in X$ such that $x<t_1<t_1'<\xi<t_2<t_2'<x'$ and $g\in G$ such that $g(t_2)<y$, $g(t_2')>y$. Let $h\in G$ sends  points  $x<t_2<t_2'<x'$ to  $x<t_1<t_1'<x'$ and identity on the points from $\sigma$. Then $h\in\st_{\sigma}$ and, hence, for any $V\in N_G(e)$ $gh\in Vg\st_{\sigma}$. Further, $g(\xi)<y$, $gh(\xi)>y$. Therefore, $gh\not\in Og\st_{\xi}$. 

It is shown that for any $\sigma\in\Sigma_X$  and  $V\in N_G(e)$ the uniform cover $\{Vg\st_{\sigma}\ |\ g\in G\}\in R_{\Sigma_X}$ is not refined in the uniform cover $\{Og\st_{\xi}\ |\ g\in G\}\in R_{\Sigma_{b_m X}}$. Hence, $L\wedge R=R_{\Sigma_X}\ne R_{\Sigma_{b_m X}}=\mathcal V_G (b_m X)$, $b_r G<\mathfrak{E} (b_r G)=e_{b_m X} G$ and $b_r G\not\in\mathbb{E}(G)$.
\end{proof}

\medskip

For the map $b_m X\to b^{\leftrightarrow}X$ of $G$-compactifications of $X$, $G=(\aut (X), \tau_{\partial})$. 
$$I^m_{\leftrightarrow}=\{f\in e_{b_m X} G\ |\ f(x)=\inf\,\vee\,\sup,\ x\in b_m X\}$$
is a nonempty compact ideal of $e_{b_m X} G$. Since $f\in e_{b_m X} G$ is a monotone map, it is easy to verify that the condition (EF) from Proposition~\ref{Reesquot} holds. The same reasoning works for the map $b^{\updownarrow}X\to  b^{+}X$ and the correspondent ideal 
$$I^{\updownarrow}_{+}=\{f\in e_{b^{\updownarrow} X} G\ |\ f(x)=\inf\,\vee\,\sup,\ x\in b^{\updownarrow} X\}$$
($\inf$ and $\sup$ in $b^{\updownarrow}X$ are correspondent images of $\inf$ and $\sup$ in $b_m X$).

\begin{cor}
Let $X$ be an ultrahomogeneous discrete chain, $G=(\aut(X), \tau_{\partial})$. 
$$e_{ b^{\leftrightarrow}X} G=e_{b_m X} G/I^m_{\leftrightarrow},\ e_{ b^+ X} G=e_{b^{\updownarrow} X} G/I^{\updownarrow}_{+}.$$
\end{cor}

\begin{df}
Let $X$ be a chain. $f$ is a partial automorphism of $X$ if $f$ is a partial bijection and $f(x)<f(y)$ whenever $x, y\in D(f)=\{z\in X\ |\ f(z)\ \mbox{is defined}\}$, $x<y$.

$J_X$ is the set of all partial automorphisms of a chain $X$.
\end{df}

\begin{lem}
Let $X$ be a chain. $J_X$ is an inverse semigroup {\rm(}multiplication is composition of maps{\rm)} and is a subsemigroup of the symmetric inverse semigroup $I_X$. 

$J_X$ is a semitopological inverse monoid with continuous inverse in the topology of pointwise convergence {\rm(}elements of $J_X$ are maps on a discrete chain $X${\rm)}.
\end{lem}

\begin{proof}
Indeed, $J_X$ is a semigroup since the composition of partial automorphisms is a partial automorphism and the composition is associative.  Since $x<y\Longleftrightarrow f(x)<f(y)$, $f\in J_X$, and the inverse element in $I_X$ to $f$ is an inverse map $f^{-1}$, $J_X$ is a subsemigroup of $I_X$. From this the last statement follows.
\end{proof}

\begin{thm}\label{sischain}
Let $X$ be an ultrahomogeneous discrete chain, $G=(\aut(X), \tau_{\partial})$. $e_{\alpha X} G$ is a sim-compactification of $G$ {\rm(}isomorphic to the semitopological inverse monoid $J_X${\rm)}.
\end{thm}

\begin{proof}
Let $\varphi: b_m X\to\alpha X$ be the map of compactifications. By Proposition~\ref{mapelliscomp} the map $\Phi: e_{b_m X} G\to e_{\alpha X} G$ of Ellis compactifications is defined. From the descriptions of $\Phi$ and $e_{b_m X} G$ it follows that $ e_{\alpha X} G$ is the set of self-maps $f$ of $\alpha X=X\cup\{\infty\}$ in the topology of pointwise convergence such that (i'') $f$ is strongly monotone (if $x<y$, then $f(x)<f(y)$) on the set $Y\subset X$, (i'') $f(x)=\infty$ if $x\in\alpha X\setminus Y$.

Since $\aut (X)$ is a subgroup of $S(X)$, $e_{\alpha X} G$ is a compact subsemigroup of $e_{\alpha X}({\rm S}(X))$. In Theorem~\ref{Roelcke precomp4-1} the isomorphism $\xi: e_{\alpha X}({\rm S}(X))\to I_X$ is defined. Its restriction $\xi|_{e_{\alpha X} G}$ to $e_{\alpha X} G$ is an isomorphism of $e_{\alpha X} G$ onto $J_X$. 
\end{proof}

\begin{rem}
{\rm For the map $b_m X\to \alpha X$ of $G$-compactifications of $X$, $G=(\aut (X), \tau_{\partial})$,  
$$I^m_{\alpha}=\{f\in e_{b_m X} G\ |\ f(x)\in b_m X\setminus X,\ x\in b_m X\}$$
is a nonempty compact ideal of $e_{b_m X} G$ by condition (ii) of item (c) of Theorem~\ref{Roelcke precomp4-2-1}. However, the condition (EF) from Proposition~\ref{Reesquot} doesn't hold.}
\end{rem}

As for $I_X$ all proper ideals of the inverse monoid $J_X$ (where $\emptyset$ is $0$) are 
$$J_X^{\delta}=\{f\in J_X\ |\ |{\rm D} (f)|\leq\delta\},\  1\leq\delta<|X|,\ J_X^{0}=\emptyset.$$ 
Ideals $J_X^{\delta}$ are invariant under the involution defined by the inverse. 

\begin{cor}\label{sischaindiscr}
In $e_{\alpha X} G$ {\rm(}$G=(\aut (X), \tau_{\partial})${\rm)}   
$$J_X^n=\{f\in e_{\alpha X} G\ |\ |D(f)|\leq n\},\ n\in\mathbb N,$$
are compact ideals. The Rees quotients $e_{\alpha X} G/\rho_{J_X^n}$, $n\in\mathbb N$, are sim-compactifications of $G$.

If $X$ is countable, then sim-compactifications $e_{\alpha X} G$ and $e_{\alpha X} G/\rho_{J_X^n}$, $n\in\mathbb N$, of $G$ are metrizable {\rm(}homeomorphic to the Cantor set{\rm)}.
\end{cor}

\begin{proof}
Compactness of $J_X^n$, $n\in\mathbb N$, can be proved in the same way as in Corollary~\ref{sisdiscr}.  $e_{\alpha X} G/\rho_{J_X^n}$, $n\in\mathbb N$, are sim-compactifications of $G$ by Lemma~\ref{ideals}. 

If $X$ is a countable set (linearly isomorphic to $\mathbb Q$), then the same arguments as in the proof of Corollary~\ref{sisdiscr} are applicable for the proof of the last statement. 
\end{proof}

\begin{thm}\label{allcompdiscr}
Let $X$ be an ultrahomogeneous discrete chain, $G=(\aut(X), \tau_{\partial})$.
$$\begin{array}{ccccccccccccc}
& & e_{b^{\leftrightarrow}X} G  &  & &  &  &  & &  & & & \\
& \nearrow & &  \searrow    & & &  &  & &  & & & \\
e_{b_m X} G &   &   &  & e_{b^{+}X} G  & \rightarrow  &  e_{\alpha X} G   & \rightarrow  & \dots & \rightarrow & e_{\alpha X} G/\rho_{J_X^n}  & \rightarrow & \dots \\
 & \searrow & e_{b^{\updownarrow}X} G &  \nearrow  &   & &  &  & &  & & & \\
\end{array}$$
If $X$ is a continuously ordered chain, then $b_r G=e_{b_m X} G$, otherwise $e_{b_m X} G>b_r G>e_{\alpha X} G$ and $b_r G$ is incomparable with $e_{b^{\leftrightarrow}X} G$, $e_{b^{\updownarrow}X} G$ and $ e_{b^{+}X} G $. 
\end{thm}

\begin{proof}
The arrows in the diagram are strict inequalities in order of compactifications of $G$. Indeed, the maps $e_{b_m X} G\to  e_{b^{\leftrightarrow}X} G$, 
$e_{b^{\updownarrow}X} G\to e_{b^{+}X} G$, $e_{\alpha X} G \to  e_{\alpha X} G/\rho_{J_X^1}$ and $e_{\alpha X} G/\rho_{J_X^n}\to e_{\alpha X} G/\rho_{J_X^{n+1}}$, $n\in\mathbb N$, are Rees quotients maps and ideals are not trivial.

The congruence on $e_{b_m X} G$ (respectively $e_{b^{\leftrightarrow}X} G$, $e_{b^{+}X} G$) defined by the map of compactifications $b_m X\to b^{\updownarrow}X$ (respectively $b^{\leftrightarrow}X\to b^{+}X$, $b^{+}X\to\alpha X$) is not trivial. 

\medskip

If $X$ is a continuously ordered chain, then by Theorem~\ref{Roelcke precomp4-2-1} $b_r G=e_{b_m X} G$. Otherwise $e_{b_m X} G>b_r G$ by Theorem~\ref{Roelcke precomp4-2-1} and Corollary~\ref{contchain}.

By Theorem~\ref{Roelcke precomp4-2-1} $b_r G$ is a compact subset of $(b_m X)^X$, $e_{\alpha X} G$ is a compact subset of $(\alpha X)^{\alpha X}$ which naturally can be identified with a compact subset of $(\alpha X)^X$. If $\varphi: b_m X\to\alpha X$ is the map of compactifications, then the restriction of the map $\Phi: (b_m X)^X\to (\alpha X)^X$ ($\Phi=\varphi^X$) to $b_r G$ is the map of compactification $b_r G\to e_{\alpha X} G$ and $b_r G>e_{\alpha X} G$. 

By Corollary~\ref{contchain} the inequalities $b_r G\leq e_{b^{+}X} G$, $b_r G\leq e_{b^{\leftrightarrow}X} G$ or $b_r G\leq e_{b^{\updownarrow}X} G$ are not valid. Therefore, to end the proof of the theorem, it remains to check that the inequality  $b_r G\geq e_{b^{+}X} G$ is not valid. 

Let $\mathcal U^+_X$ be the subspace uniformty on $X$ as subset of $b^{+} X$. Take $x, y\in X$ and ${\rm U}\in \mathcal U^+_X$ such that $\exists\ z\in X$, $y<z$, $(y, z)\not\in {\rm U}$. 

Fix ${\rm W}=\{(h, g)\in G\times G\ |\ (g(x), h(x))\in {\rm U}\}\in\mathcal V_G (b^{+} X)$.  Ultrahomogeneity of $X$ yields that $\forall\ O\in N_G (e)$ $\exists\ g\in G$ and $\exists\ h\in O$ such that $x<h(x)$, $g(x)=y$, $g(hx)=z$. Therefore, $gh\in OgO$, however, $(g(x), (gh)(x))\not\in {\rm U}$. Hence, $ V_G  (b^{+} X)\not\subset L\wedge R$ and $b_r G\ngeq e_{b^{+}X} G$.
\end{proof}

\begin{rem}\label{sscompdiscr}
{\rm If $X$ is not a continuously ordered chain, then $e_{b_m X} G>b_r G$ and by~\cite[Remark 3.14]{KozlovLeiderman} $e_{b_m X} G$ is not a sm-compactification of $G$. $e_{b^{\leftrightarrow}X} G$, $e_{b^{\updownarrow}X} G$ and $e_{b^{+}X} G$ are Ellis compactificaitons incomparable with  $b_r G$. Hence, they are not sm-compactifications of $G$.

Let $X$ be a continuously ordered chain.

(a) $e_{b_m X} G=b_r G$ but $b_r G$ is not a semitopological semigroup. Indeed, take the map $f\in b_r G$,  $f(\inf)=\inf$, $f(t)=\sup$ otherwise. Any nbd of the map $g\in b_r G$, $g(\sup)=\sup$, $g (t)=\inf$ otherwise, is of the form  $O_x^y=\{h\in b_r G\ |\ h((x, 0))<(y, 0)\}$, $x, y\in X$. Since $f\circ g=g$, in any nbd of $g$ of the form of $O_x^y$ there is $h\in b_r G$ such that $h(\inf)=\inf$, $h(t)=(y, -1)$, $\inf<t<\sup$,  $h(\sup)=\sup$. Hence, $(f\circ h)(\inf)=\inf$, $(f\circ h)(t)=\sup$ otherwise and can't belong to the arbitrary nbd of $g$ of the form $O_x^y$. Thus, multiplication on the left in $b_r G$ is not continuous.

(b) $e_{b^{\leftrightarrow}X} G$ is not a semitopological semigroup. Indeed, take $f\in e_{b^{\leftrightarrow}X} G$,  $f((x, -1))=f((x, 0))=f((x, 1))=(x, 1)$, $x$ is a point in $X$, $f(t)=\infty$ otherwise. Any nbd of $f$ is of the form $O_{y, z}^{s, \tau, r}=\{h\in e_{b^{\leftrightarrow}X} G\ |\ h((y, +1))<(s, 0),\ h((x, -1))<(\tau, 0),\ h((x, +1))<(\tau, 0),\ h((z, -1))>(r, 0)\}$, $y, z, \tau, s, r\in X$, $y<x<z$, $s<x<\tau<r$. Since $f\circ f=f$, in any nbd of $f$ of the form of $O_{y, z}^{s, \tau, r}$ there exists $h\in e_{b^{\leftrightarrow}X} G$ such that $(\tau, 0)>h(x, -1)=h(x, 0)=h(x, 1)>(x, 1)$, $h (t)=\infty$ otherwise. $f\circ h$ is a constant map equal to $\infty$. Hence, $f\circ h$ can't belong to the arbitrary nbd of $f$ of the form  $O_{y, z}^{s, \tau, r}$. Thus, multiplication on the left in $e_{b^{\leftrightarrow}X} G$ is not continuous.

(c) $e_{b^{\updownarrow}X} G$ is not a semitopological semigroup. The same map as in (a) can be used. 

(d) $e_{b^{+}X} G$ is not a semitopological semigroup. The same map as in (b) can be used.}
\end{rem}


\begin{rem}
{\rm Let $X$ be an ultrahomogeneous chain, $G$ is a subgroup of $\aut (X)$ which acts ultratransitively on $X$. If $X$ is LOTS, then $(G, \tau_p)$ is a dense subgroup of $(\aut (X), \tau_p)$, if $X$ is a discrete chain, then $(G, \tau_{\partial})$ is a dense subgroup 
of $(\aut (X), \tau_{\partial})$  by Remark~\ref{denseultr}.  The results of \S~\ref{ultLOT} and \S~\ref {ultLOTa} concerning $G$-compactifications of $X$ and Ellis compactifications of $G$ are valid (except that the number of $G$-compactifications of $X$ and, hence, Ellis compactifications of $G$ may increase). Moreover, the corresponding compactifications of $\aut (X)$ and $G$ are homeomorphic (and topologically isomorphic if they are semigroups). }
\end{rem}

\begin{ex}\label{Thompson}
{\rm The action of the Thompson group $(F, \tau_{\partial})$ is ultratransitive on the discrete dense chain $\mathbb Q_r$ of dyadic rational numbers of $(0, 1)$,  and every automorphism from $F$ is a piecewise-linear map that have a finite number of breakpoints~\cite{Canon}. Also, the action of $(F, \tau_p)$ is ultratransitive on LOTS $\mathbb Q_r$. 

$G$-compactifications of the discrete chain $\mathbb Q_r$ for the action $(F, \tau_{\partial})\curvearrowright\mathbb Q_r$ and LOTS $\mathbb Q_r$ for the action $(F, \tau_p)\curvearrowright\mathbb Q_r$ are the same as for the action of the group $\aut (\mathbb Q_r)$. This follows from the fulfillment of the condition:

for any proper gap $\nu=(A, B)$ ($A, B\subset \mathbb Q_r$, $A, B\ne\emptyset$, $A\cap B=\emptyset$, $A\cup B=\mathbb Q_r$, if $x\in A$, $y\in B$, then $x<y$) $\st_{\nu}$ is a closed subgroup of $(F, \tau_{\partial})$ (respectively  $(F, \tau_p)$) and 
$$\mbox{the actions}\ \st_{\nu}\curvearrowright A,\ \st_{\nu}\curvearrowright B\ \mbox{are ultratransitive}.$$
Therefore, results from Section~\ref{chain} can be applied.}
\end{ex}

\end{document}